\documentclass[a4paper]{amsart}
\usepackage[utf8]{inputenc}
\usepackage{tikz,eucal,hyperref}
\usetikzlibrary{matrix,arrows,shapes,calc,intersections,decorations.markings,decorations.pathmorphing}

\usepackage{tikz-cd}
\usepackage{amssymb,mathtools}
\usepackage{lmodern}
\usepackage[T1]{fontenc}
\usepackage{enumitem}

\tikzset{bul/.style={circle,draw,fill,inner sep=1.2pt}}
\tikzset{wbul/.style={circle,draw=white,fill=white,inner sep=1.2pt}}
\tikzset{tbul/.style={circle,fill,inner sep=0.8pt}}
\tikzset{rbul/.style={circle,fill=red,draw=red,opacity=0.5,inner sep=1pt}}
\tikzset{bbul/.style={circle,fill=blue,draw=blue,inner sep=1.5pt}}
\tikzset{tv/.style={rectangle,draw=black,fill=white,inner sep=2pt}}
\tikzset{anti/.style={rectangle,fill=black, sloped,inner xsep=0.5pt, inner ysep=2pt}}
\tikzset{pok/.style={rectangle,fill=black, draw, sloped,inner xsep=0pt, inner ysep=1.2pt}}
\tikzset{obul/.style={circle,fill=orange,draw=orange,opacity=0.5,inner sep=1pt}}

\tikzset{->-/.style={decoration={
  markings,
  mark=at position #1 with {\arrow{>}}},postaction={decorate}}}

\theoremstyle{plain}
  \newtheorem{thm}{Theorem}[section]
  \newtheorem{defn}[thm]{Definition}
  \newtheorem{thd}[thm]{Theorem/Definition}
  \newtheorem*{defn*}{Definition}
  \newtheorem{prop}[thm]{Proposition}
  \newtheorem{cor}[thm]{Corollary}
  \newtheorem{lem}[thm]{Lemma}
\theoremstyle{definition}
  \newtheorem{example}[thm]{Example}
  \newtheorem{rem}[thm]{Remark}

\newcommand{\mf}{\mathfrak}
\newcommand{\mc}{\mathcal}
\newcommand{\ms}{\mathsf}
\newcommand{\on}{\operatorname}

\newcommand{\g}{\mathfrak{g}}
\newcommand{\dd}{\mathfrak{d}}
\newcommand{\h}{\mathfrak{h}}
\newcommand{\p}{\mathfrak{p}}

\newcommand{\C}{\mathcal C}

\newcommand{\Set}{\mathsf{Set}}

\newcommand{\Hom}{\operatorname{Hom}}
\newcommand{\la}{\langle}
\newcommand{\ra}{\rangle}
\newcommand{\R}{\mathbb{R}}
\newcommand{\Q}{\mathbb{Q}}

\newcommand{\id}{\on{id}}

\newcommand{\uln}[1]{\underline{\smash{#1}}}

\title{Quantization of Poisson Hopf algebras}
\author{Ján Pulmann}
\author{Pavol \v{S}evera}
\address{Section of Mathematics, Universit\'{e} de Gen\`{e}ve, Geneva, Switzerland}
\email{jan.pulmann@unige.ch}
\email{pavol.severa@gmail.com}
\thanks{Supported by the NCCR SwissMAP of the Swiss National Science Foundation.}

\begin{document}
\maketitle

\begin{abstract}
We describe a method for quantization of Poisson Hopf algebras in  $\mathbb Q$-linear symmetric monoidal categories. It is compatible with tensor products and can also be used to produce braided Hopf algebras. The main idea comes from the fact that nerves of groups are symmetric simplicial sets. Nerves of Hopf algebras then turn out to be braided rather than symmetric and nerves of Poisson Hopf algebras to be infinitesimally braided. The problem is thus solved via the standard machinery of Drinfeld associators. 
\end{abstract}

\tableofcontents

\section{Introduction}

One of the most natural deformation quantization problems is quantization of Poisson Hopf algebras
\cite{drqg}, i.e.\ of commutative Hopf algebras with a compatible Poisson bracket. Given a Poisson Hopf algebra with product $m_0$, coproduct $\Delta_0$, antipode $S_0$, Poisson bracket $p$, unit $\eta$ and counit $\epsilon$, the problem is to deform $m_0$, $\Delta_0$, and $S_0$ to
$$m_{}=\sum_{n=0}^\infty \hbar^n m_n\qquad \Delta_{}=\sum_{n=0}^\infty \hbar^n \Delta_n\qquad S_{}=\sum_{n=0}^\infty \hbar^n S_n$$
so that the result is still a Hopf algebra and so that
$$m_{}-m_{}^\mathit{op}=\hbar\,p+O(\hbar^2).$$

The deformation should be given by universal formulas with rational coefficients, i.e.\ as a morphism of props
$$\ms{Hopf}\to\ms{PoissHopf}(\Q)_\hbar,$$
 and thus usable for quantization of Poisson Hopf algebras in arbitrary $\Q$-linear symmetric monoidal categories and functorial in a rather strong sense. Furthermore, it is natural to demand the quantization to be compatible with tensor products of Hopf algebras; equivalently it means that the morphism of props is compatible with a suitable cocommutative coalgebra enrichment.

Quantization of Poisson Hopf algebras includes, in particular, quantization of Lie bialgebras, which was solved in the seminal work of Etingof and Kazhdan \cite{EK} (other approaches include \cite{Enr, Tam, ja}). Nonetheless, the problem of quantization of Poisson Hopf algebras is stronger and appears more natural (in special cases of Poisson Hopf algebras a quantization of Lie bialgebras is sufficient; see \cite{EK2} for the case of Poisson algebraic groups).
 
We solve this problem using nerves of Hopf algebras. The nerve of a group is not just a simplicial set, but a symmetric simplicial set. A similar statement is true for nerves of commutative Hopf algebras. If we replace the symmetric structure with a braided one, we get nerves of (possibly braided) Hopf algebras. An infinitesimally braided structure corresponds to nerves of Poisson Hopf algebras, and their quantization can thus be obtained by Drinfeld associators (which are machines producing braided structures out of infinitesimally braided ones).

The natural setup for our method is slightly more general: quantization of Poisson Hopf algebras in infinitesimally braided categories to braided Hopf algebras. At the end of the paper we also describe the outcome of our method when applied to quantization of suitable Poisson Hopf algebroids.

\section{Preliminaries}

\subsection{Algebras, coalgebras, Hopf algebras}

An \emph{algebra} (or a \emph{monoid}) in a monoidal category $\C$ is an object $A\in\C$ equipped with an associative product $m\colon A\otimes A\to A$ and with a unit $\eta\colon 1_\C\to A$ of $m$. Similarly a \emph{coalgebra} $C\in\C$ is equipped with a coassociative coproduct $\Delta\colon C\to C\otimes C$ and with a counit $\epsilon\colon C\to 1_\C$.

An algebra $A$ in a braided monoidal category (BMC) $\C$ is \emph{commutative} if the product $m\colon A\otimes A\to A$ satisfies $m=m\circ\beta_{A,A}$, where $\beta_{X,Y}:X\otimes Y\to Y\otimes X$ denotes the braiding, i.e.
$$
\begin{tikzpicture}[baseline=0.5cm,xscale=0.8,yscale=0.5]
\node[above](E) at (0,2.5){$A$};
\coordinate(P) at (0,2);
\node[below](f) at (-0.5,0){$A$};
\node[below](g) at (0.5,0){$A$};
\draw[thick] (g)to[out=90,in=-80](P);
\draw[thick] (f)to[out=90,in=-100](P);
\draw[thick] (P)--(E);
\node at (1,1.2){=};
\begin{scope}[xshift=2cm]
\node[above](E) at (0,2.5){$A$};
\coordinate(P) at (0,2);
\coordinate(L) at (0.35,1.2);
\coordinate(R) at (-0.35,1.2);
\node[below](f) at (-0.5,0){$A$};
\node[below](g) at (0.5,0){$A$};
\draw[thick] (g)to[out=90,in=-90](R);
\draw[line width=1ex,white](f)to[out=90,in=-90](L);
\draw[thick] (f)to[out=90,in=-90](L);
\draw[thick](L)to[out=90,in=-80](P) (R)to[out=90,in=-100](P);
\draw[thick] (P)--(E);
\end{scope}
\end{tikzpicture}
$$
(with morphisms going upwards and braiding depicted by
$\begin{tikzpicture}[scale=0.25,baseline=0.7]
\draw[thick] (1,0)--(0,1);
\draw[line width=1ex,white] (0,0)--(1,1);
\draw[thick] (0,0)--(1,1);
\end{tikzpicture}$).

If $A_1,A_2$ are algebras  in a BMC $\C$, then so is $A_1\otimes A_2$ with the product \cite[Lem.~2.1]{Maj}
$$\begin{tikzpicture}[xscale=0.7,yscale=0.5,baseline=0.5cm]
\coordinate (diff) at (0.75,0);
\coordinate (dy) at (0,0.5);
\coordinate(A1) at (0,0);
\coordinate(B1) at ($(A1)+(diff)$);
\coordinate(A2) at (2,0);
\coordinate(B2) at ($(A2)+(diff)$);
\coordinate(A3) at (1,2.5);
\coordinate(B3) at ($(A3)+(diff)$);
\coordinate(ep) at (0.3,0);
\coordinate(A) at ($(A3)-(dy)$);
\coordinate(B) at ($(B3)-(dy)$);

\node[below] at (A1){$A_1$};
\node[below] at (B1){$A_2$};
\node[below] at (A2){$A_1$};
\node[below] at (B2){$A_2$};
\node[above] at (A3){$A_1$};
\node[above] at (B3){$A_2$};

\draw[thick] (A2)to[out=90,in=-80](A);
\draw[line width=1ex,white] (B1)to[out=90,in=-100](B);
\draw[thick] (B1)to[out=90,in=-100](B);
\draw[thick] (A1)to[out=90,in=-100](A);
\draw[thick] (B2)to[out=90,in=-80](B);
\draw[thick] (B)--(B3);
\draw[thick](A)--(A3);
\end{tikzpicture}
$$
and with the unit 
$1_\C\cong1_\C\otimes1_\C\xrightarrow{\eta_{A_1}\otimes\eta_{A_2}}A_1\otimes A_2.$
In this way algebras in $\C$ form a monoidal category.
The tensor product of two commutative algebras in a symmetric monoidal category (SMC) is commutative, but in a BMC it is, in general, not true.

A coalgebra $H$ in the monoidal category of algebras in a BMC $\C$ is a (braided) \emph{bialgebra}.
 Explicitly this means that $H$ is both an algebra and a coalgebra in $\C$ and that we have the identities
$$
\begin{tikzpicture}[scale=0.5,baseline=0.7cm]
\coordinate(in1) at (-1,0);
\coordinate(in2) at (1,0);
\coordinate(m1) at (0,1);
\coordinate(m2) at (0,2);
\coordinate(out1) at (-1,3);
\coordinate(out2) at (1,3);
\draw[thick](in1)--(m1)--(in2) (m1)--(m2) (out1)--(m2)--(out2);
\end{tikzpicture}
=
\begin{tikzpicture}[scale=0.5,baseline=0.7cm]
\coordinate(in1) at (-1,0);
\coordinate(in2) at (1,0);
\coordinate(m1L) at (-0.7,0.8);
\coordinate(m2L) at (-0.7,2.2);
\coordinate(m1R) at (0.7,0.8);
\coordinate(m2R) at (0.7,2.2);
\coordinate(out1) at (-1,3);
\coordinate(out2) at (1,3);
\draw[thick](m1R)--(m2L);
\draw[line width=1ex,white] (m1L)--(m2R);
\draw[thick](m1L)--(m2R);
\draw[thick](in1)--(m1L)--(m2L)--(out1) (in2)--(m1R)--(m2R)--(out2);
\end{tikzpicture}
$$
and 
$$
\epsilon\circ\eta=\on{id}_{1_\C},\quad \Delta\circ(\eta\otimes\eta)=\eta,\quad \epsilon\circ m=\epsilon\otimes\epsilon
$$
(where we tacitly identify $1_\C\otimes 1_\C$ with $1_\C$ via the canonical isomorphism).

A (braided) \emph{Hopf algebra} in a BMC $\C$ is a bialgebra $H$ with an additional morphism $S\colon H\to H$ (the \emph{antipode}) such that 
$$
m\circ(\on{id}\otimes S)\circ\Delta=m\circ(S\otimes\on{id})\circ\Delta=\eta\circ\epsilon.
$$


\subsection{Lax monoidal functors}

A \emph{lax monoidal functor} is a functor $F\colon \C_1\to\C_2$ between two monoidal categories, equipped with natural transformations (called \emph{coherence morphisms})
$$c^F_{X,Y}\colon F(X)\otimes F(Y)\to F(X\otimes Y),\quad c^F_1\colon 1_{\C_2}\to F(1_{\C_1})$$
making $F$ compatible with the associativity and the unit constraints of $\C_1$ and $\C_2$ (see e.g.\ \cite[Sec.~XI.2]{McL} for details).
Coherence morphisms can be graphically depicted 
as
$$c^F_{X,Y}\ =\!
\begin{tikzpicture}[yscale=0.7, baseline=0.7cm]
\coordinate (A1) at (-1,0);
\coordinate (B1) at (-0.5,2);
\coordinate (A2) at (1,0);
\coordinate (B2) at (0.5,2);
\coordinate (C1) at (-0.3,0);
\coordinate (C2) at (0.3,0);
\coordinate (X) at (-0.65,0);
\coordinate (Y) at (0.65,0);

\fill[black!10] (A1) to[out=90,in=-90] (B1)--(B2) to[out=-90,in=90](A2) -- (C2)..controls +(0,0.7) and +(0,0.7)..(C1);

\draw[black!30,thick] (A1) to[out=90,in=-90] (B1) (B2) to[out=-90,in=90](A2)  (C2)..controls +(0,0.7) and +(0,0.7)..(C1);

\draw[thick] (X) to[out=90,in=-90] (-0.15,2);
\draw[thick] (Y) to[out=90,in=-90] (0.15,2);

\node[below] at (X) {$F(X)$};
\node[below] at (Y) {$F(Y)$};
\node[below] at (0,0){${\otimes}\vphantom{(}$};
\node[above] at (0,2) {$F(X\otimes Y)$};
\end{tikzpicture}\qquad
c^F_1\ =
\begin{tikzpicture}[yscale=0.7, baseline=0.7cm]
\fill[black!10] (4,2)..controls +(0,-2.5) and +(0,-2.5)..(5,2);
\draw[black!30,thick] (4,2)..controls +(0,-2.5) and +(0,-2.5)..(5,2);
\node[below] at (4.5,0) {$1_{\C_2}$};
\node[above] at (4.5,2) {$F(1_{\C_1})$};
\end{tikzpicture}
$$

A \emph{strong monoidal functor} is a lax monoidal functor whose coherence morphisms are isomorphisms. A \emph{strict monoidal functor} is a strong monoidal functor whose coherence morphisms are identities.

A natural transformation $\alpha_X:F(X)\to G(X)$ between two lax monoidal functors $F,G\colon\C_1\to\C_2$ is \emph{monoidal} if 
$$\alpha_{X\otimes Y}\circ c^F_{X,Y}= c^G_{X,Y}\circ(\alpha_X\otimes\alpha_Y) \quad\text{and}\quad \alpha_{1_{\C_2}}\circ c^F_1=c^G_1.$$
Lax monoidal functors form a category with monoidal natural transformations as morphisms.

A lax monoidal functor $F\colon\C_1\to\C_2$ between two BMCs is \emph{braided} if $$F(\beta^{\C_1}_{X,Y})\circ c^F_{X,Y}=c^F_{Y,X}\circ\beta^{\C_2}_{F(X),F(Y)}.$$

Lax monoidal functors send algebras to algebras:  If $A\in\C_1$ is an algebra then $F(A)\in\C_2$ is an algebra with the product  
$$F(m)\circ c^F_{A,A}\ =\!
\begin{tikzpicture}[yscale=0.7,baseline=0.7cm]
\coordinate (A1) at (-1,0);
\coordinate (B1) at (-0.35,2);
\coordinate (A2) at (1,0);
\coordinate (B2) at (0.35,2);
\coordinate (C1) at (-0.3,0);
\coordinate (C2) at (0.3,0);
\coordinate (X) at (-0.65,0);
\coordinate (Y) at (0.65,0);

\fill[black!10] (A1) to[out=90,in=-90] (B1)--(B2) to[out=-90,in=90](A2) -- (C2)..controls +(0,0.7) and +(0,0.7)..(C1);

\draw[black!30,thick] (A1) to[out=90,in=-90] (B1) (B2) to[out=-90,in=90](A2)  (C2)..controls +(0,0.7) and +(0,0.7)..(C1);

\draw[thick] (X) to[out=90,in=-100] (0,1.6);
\draw[thick] (Y) to[out=90,in=-80] (0,1.6)--(0,2);

\node[below] at (X) {$F(A)$};
\node[below] at (Y) {$F(A)$};
\node[below] at (0,0){${\otimes}\vphantom{(}$};
\node[above] at (0,2) {$F(A)$};
\end{tikzpicture}
$$
and with the unit $F(\eta)\circ c^F_1$.
Braided lax monoidal functors send commutative algebras to commutative algebras.

\subsection{Props}
A \emph{prop} is a  strict SMC whose objects are symbols $\bullet^n$, $n\geq 0$,  with the tensor product of objects $\bullet^m\otimes\bullet^n=\bullet^{m+n}$ and with the unit object $1=\bullet^0$. A \emph{braided prop} is a  strict BMC with the same property. 

A morphism of (braided or ordinary) props $F\colon \mc P_1\to\mc P_2$ is a strict braided  monoidal functor such that $F(\bullet)=\bullet$.

If $\mc P$ is an ordinary/braided prop, a \emph{$\mc P$-algebra} in a strict SMC/BMC $\C$ is an object $A\in\C$ together with a strict braided monoidal functor $F\colon\mc P\to\C$ such that $F(\bullet)=A$. 

\begin{rem}\label{rem:parfirst}
	If $\C$ is not strict, $F$ is rather a strong braided monoidal functor,  unique up to  monoidal isomorphisms which are the identity at the object $\bullet$. We can specify a particular $F$ by demanding $F(\bullet^n)=A^{\otimes n}$ for a chosen parenthesization of the tensor power $A^{\otimes n}$, and demanding $c^F_1=\id$ and $c^F_{\bullet^m,\bullet^n}\colon A^{\otimes m}\otimes A^{\otimes n}\to A^{\otimes (m+n)} $ to be the corresponding composition of the associativity (or unit) constraints of $\C$.
\end{rem}

\begin{example}\label{ex:Com}
Let $\ms{Com}$ be the prop of commutative algebras. It is generated by morphisms $\uln m\colon {\bullet\bullet} \to\bullet$ and $\uln \eta\colon 1\to\bullet$ modulo the relations saying that $\uln m$ is commutative and associative and that $\uln\eta$ is a unit for $\uln m$.
A $\ms{Com}$-algebra in a SMC is a commutative algebra.

A morphism $\bullet^m\to\bullet^n$ in $\ms{Com}$ can be represented graphically as
$$
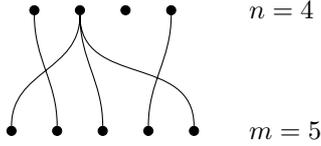
\begin{tikzpicture}[xscale=0.6,yscale=0.8]
\node[bul](a1) at (0,0){};
\node[bul](a2) at (1,0){};
\node[bul](a3) at (2,0){};
\node[bul](a4) at (3,0){};
\node[bul](a5) at (4,0){};
\node at (5,0) [right] {$m=5$};

\node[bul](b1) at (0.5,2){};
\node[bul](b2) at (1.5,2){};
\node[bul](b3) at (2.5,2){};
\node[bul](b4) at (3.5,2){};
\node at (5,2) [right] {$n=4$};

\draw(a1)to[out=90,in=-90](b2) (a2)to[out=90,in=-90](b1) (a3)to[out=90,in=-90](b2) (a4)to[out=90,in=-90](b4) (a5)to[out=90,in=-90](b2);
\end{tikzpicture}
$$
with joining lines corresponding to the product ${\bullet\bullet}\to\bullet$ and missing lines to the unit $1\to\bullet$,
i.e.\ as a map $\{1,\dots,m\}\to\{1,\dots,n\}$. Composition of morphisms is then equal to the composition of these maps.
\end{example}

\begin{example}

Let $\ms{BrCom}$ be the braided prop of commutative algebras. It is again generated by two morphisms $\uln m\colon{\bullet\bullet}\to\bullet$ (the commutative product) and $\uln\eta\colon1\to\bullet$ (the unit). A $\ms{BrCom}$-algebra in a BMC is a commutative algebra.

 The morphisms $\bullet^m\to\bullet^n$ in $\ms{BrCom}$ are braids with $m$ strands, non-bijectively attached at the top: 
$$
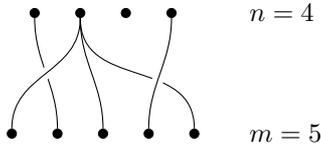
\begin{tikzpicture}[xscale=0.6,yscale=0.8]
\node[bul](a1) at (0,0){};
\node[bul](a2) at (1,0){};
\node[bul](a3) at (2,0){};
\node[bul](a4) at (3,0){};
\node[bul](a5) at (4,0){};
\node at (5,0) [right] {$m=5$};

\node[bul](b1) at (0.5,2){};
\node[bul](b2) at (1.5,2){};
\node[bul](b3) at (2.5,2){};
\node[bul](b4) at (3.5,2){};
\node at (5,2) [right] {$n=4$};

\draw (a2)to[out=90,in=-90](b1);
\draw[line width=1ex,white](a1)to[out=90,in=-90](b2);
\draw(a1)to[out=90,in=-90](b2);
\draw (a3)to[out=90,in=-90](b2);
\draw (a5)to[out=90,in=-90](b2);
\draw[line width=1ex,white] (a4)to[out=90,in=-90](b4);
\draw (a4)to[out=90,in=-90](b4);
\end{tikzpicture}
$$

\end{example}

\begin{rem}
	The endomorphism monoids $\ms{BrCom}(\bullet^n, \bullet^n)$ appeared in \cite{Lav}, under the name of $n$-\textit{vines}.
	
	In \cite{FieLo}, related categories $S_{*}$ and $B_*$ were studied, as examples of \textit{crossed simplicial groups}. The  difference is that our $\bullet$ is a \textit{commutative} algebra, i.e. the unique factorization property \cite[Def.~1.1(d)]{FieLo} does not hold for $\ms{Com}$ and $\ms{BrCom}$.
\end{rem}

%
%
%

\section{Nerves of  Hopf algebras}

\subsection{Symmetric simplicial spaces}

An (augmented) symmetric simplicial object in a category $\mc S$ is a functor 
$$\ms{Com}^\text{op}\to\mc S.$$
 Informally speaking, we want to study the case when $\mc S$ is a ``category of spaces''. We take the dual point of view and replace ``spaces'' with commutative algebras. Hence, if $\C$ is a symmetric monoidal category and $\on{ComAlg}(\C)$ the category of commutative algebras in $\C$, our object of interest is the category of functors 
$$\ms{Com}\to\on{ComAlg}(\C),$$
i.e.\ (augmented) symmetric cosimplicial commutative algebras in $\C$.

Let us observe that every object $\bullet^n$ of $\ms{Com}$ is a commutative algebra, being  a tensor power of the commutative algebra $\bullet$, and also that every morphism in $\ms{Com}$ is a morphism of algebras. This gives us the following result.

\begin{prop}\label{prop:scosca}
Suppose $\C$ is a symmetric monoidal category. The category of functors
$$\ms{Com}\to\on{ComAlg}(\C)$$
is isomorphic to the category of symmetric lax monoidal functors
$$\ms{Com}\to\C$$
with monoidal natural transformations as morphisms.

Namely, if $F\colon \ms{Com}\to\C$ is symmetric lax monoidal then $F(\bullet^n)$ is a commutative algebra in $\C$ as $\bullet^n$ is a commutative algebra  in $\ms{Com}$, and thus $F$ becomes a functor $\ms{Com}\to\on{ComAlg}(\C)$.
\end{prop} 
\begin{proof}
To get the isomorphism in the opposite direction, if $F\colon \ms{Com}\to\on{ComAlg}(\C)$ is a functor then we can define a symmetric lax monoidal structure on $F$  as follows: $c^F_{X,Y}$ is the composition,
$$F(X)\otimes F(Y)\to F(X\otimes Y)\otimes F(X\otimes Y)\to F(X\otimes Y)$$
where the first arrow comes  from 
$$X=X\otimes 1\xrightarrow{\on{id}_X\otimes\eta_Y} X\otimes Y\quad\text{and}\quad Y=1\otimes Y\xrightarrow{\eta_X\otimes\on{id}_Y} X\otimes Y $$
and the second arrow is the product in $F(X\otimes Y)$,
and $c^F_1$ is the unit $1_\C\to F(1)$ of the commutative algebra $F(1)$.
\end{proof}

\subsection{Nerves of groups and  groupoids}

If $S$ is a set, let $\mathit{Pair}_S$ denote the pair groupoid of $S$, i.e.\ the groupoid with the set of objects $S$ and with a unique morphism between any two objects.

The (symmetric augmented) \emph{nerve} of a groupoid $G$ is the functor 
$$N_G\colon \ms{Com}^\text{op}\to\Set$$
 given on objects by
$$N_G(\bullet^n)=\Hom(\mathit{Pair}_{\{1,2,\dots,n\}},G).$$
On morphisms it is given by seeing them as maps $\{1,2,\dots,m\}\to\{1,2,\dots,n\}$.

If $G$ is a group, this means
$$N_G(\bullet^n)=\bigl\{g\colon \{1,2,\dots,n\}^2\to G\mid g(i,i)=1 \text{ and } g(i,j)\,g(j,k)=g(i,k)\bigr\}$$
and we have a bijection 
\begin{equation}\label{groupNerveIso}
N_G(\bullet^n)\cong G^{n-1},\quad g\mapsto\bigl(g(i,i+1)\mid 1\leq i\leq n-1\bigr).
\end{equation}

We can recognize nerves of groups among all functors $N\colon\ms{Com}^\text{op}\to\ms{Set}$ using \eqref{groupNerveIso} as follows. Let us consider the morphisms $\phi_i\colon{\bullet\bullet}\to\bullet^n$
$$\phi_i=\ %
\begin{tikzpicture}[scale=0.5,baseline=0.2cm]
\draw (0,0) node[bul]{} (1,0) node[bul] {};
\draw (0,1) node[bul]{} node[above]{\scriptsize $\mathstrut i$} (1,1) node[bul] {} node[above]{\scriptsize $\mathstrut i+1$};
\draw (-2,1) node[bul] {} node[above] {\scriptsize $\mathstrut 1$} (-1,1) node {$\cdots$};
\draw (3,1) node[bul] {} node[above] {\scriptsize $\mathstrut n$} (2,1) node {$\cdots$};
\draw (0,0)--(0,1) (1,0)--(1,1);
\end{tikzpicture}
\qquad (1\leq i\leq n-1)
$$
and the resulting map 
\begin{equation}\label{recnerve}
\bigl(N(\phi_1),\dots,N(\phi_{n-1})\bigr)\colon N(\bullet^{n})\to N({\bullet\bullet})^{n-1}.
\end{equation}
 We then have the following elementary and well known result.
\begin{prop}\label{prop:grnerve}
The assignment $G\mapsto N_G$ is an equivalence of categories between the category of groups and the category of functors $N\colon\ms{Com}^\mathrm{op}\to\Set$ such that \eqref{recnerve} is a bijection for every $n\geq 2$ and such that $|N(1)|=|N(\bullet)|=1$.

The group corresponding to $N$ is $G=N({\bullet\bullet})$, the product is
$$N\Bigl(\,
\begin{tikzpicture}[xscale=0.35,yscale=0.6, baseline=0.2cm]
\node(a) [bul] at (0,0){};
\node(b) [bul]at (1,0){};
\node(1) [bul]at (-0.5,1){};
\node(2) [bul]at (0.5,1){};
\node(3) [bul]at (1.5,1){};
\draw (a)to[out=100,in=-80](1) (b)to[out=80,in=-100](3);
\end{tikzpicture}\,\Bigr)\colon G\times G\to G
$$
(where we identify $N(\bullet^3)$ with $G\times G$ via \eqref{recnerve})
 and the unit and the inverse are
$$
N\Bigl(\,
\begin{tikzpicture}[xscale=0.35,yscale=0.6, baseline=0.2cm]
\node(a) [bul] at (0,0){};
\node(b) [bul]at (1,0){};
\node(1) [bul]at (0.5,1){};
\draw (a) to[out=80, in=-90] (1) (b) to[out=100, in=-90] (1);
\end{tikzpicture}\,\Bigr)\quad\text{and}\quad
N\Bigl(\,
\begin{tikzpicture}[xscale=0.35,yscale=0.6, baseline=0.2cm]
\node(a) [bul] at (0,0){};
\node(b) [bul]at (1,0){};
\node(1) [bul]at (0,1){};
\node(2) [bul]at (1,1){};
\draw (a) to[out=75, in=-105] (2) (b) to[out=105, in=-75] (1);
\end{tikzpicture}\,\Bigr).
$$
\end{prop}

\subsection{Nerves of commutative Hopf algebras}
By analogy with groups one can define the nerve of a commutative Hopf algebra $H\in\C$ where $\C$ is a symmetric monoidal category (SMC). It is a functor
$$N_H\colon \ms{Com}\to\on{ComAlg}(\C)$$ or equivalently, in view of Proposition \ref{prop:scosca}, a symmetric lax monoidal functor
$$N_H\colon \ms{Com}\to\C.$$
On objects it is given by $N_H(\bullet^ n)=H^{\otimes(n-1)}$, $N_H(1)=1_\C$ (to simplify the notation we suppose that $\C$ is strict monoidal; if not, one needs to choose a parethesization of $H^{\otimes(n-1)}$ and use the associativity constraint of $\C$ accordingly).

The functor $N_H$ is defined on morphisms via a graphical algorithm, which takes a morphism $\phi\colon\bullet^m\to\bullet^n$ in $\mathsf{Com}$ and outputs a morphism $H^{\otimes (m-1)} \to H^{\otimes (n-1)}$ given in terms of the structure operations of the Hopf algebra:
$$
\begin{tikzpicture}[xscale=0.6,yscale=0.8,baseline=0.8cm]
\node[bul](a1) at (0,0){};
\node[bul](a2) at (1,0){};
\node[bul](a3) at (2,0){};
\node[bul](a4) at (3,0){};
\node[bul](a5) at (4,0){};

\node[bul](b1) at (0.5,2){};
\node[bul](b2) at (1.5,2){};
\node[bul](b3) at (2.5,2){};
\node[bul](b4) at (3.5,2){};

\draw(a1)to[out=90,in=-90](b2) (a2)to[out=90,in=-90](b1) (a3)to[out=90,in=-90](b2) (a4)to[out=90,in=-90](b4) (a5)to[out=90,in=-90](b2);
\end{tikzpicture}
\qquad
\xmapsto{N_H}
\qquad
\begin{tikzpicture}[xscale=0.6,yscale=0.8,baseline=0.8cm]
\node[bul,lightgray](a1) at (0,0){};
\node[bul,lightgray](a2) at (1,0){};
\node[bul,lightgray](a3) at (2,0){};
\node[bul,lightgray](a4) at (3,0){};
\node[bul,lightgray](a5) at (4,0){};

\node[bul,lightgray](b1) at (0.5,2){};
\node[bul,lightgray](b2) at (1.5,2){};
\node[bul,lightgray](b3) at (2.5,2){};
\node[bul,lightgray](b4) at (3.5,2){};

\coordinate(m1) at (1,1.5);
\coordinate(m2) at (2,1.5);
\coordinate(m3) at (3,1.5);
\coordinate(c1) at (2.5,0.5);
\coordinate(c2) at (3.5,0.5);

\draw[thick] (0.5,0)--node[anti]{}(0.5,0.5)to[out=90, in=-120](m1) (1.5,0)--(1.5,0.5)to[out=90, in=-60](m1) (m1)--(1,2)%
  (2.5,0)--(c1)to[out=120,in=-120] (m2)--(2,2) (c1)to[out=60,in=-120] (m3)--(3,2)%
  (3.5,0)--node[anti]{}(c2)to[out=120,in=-60] (m2) (c2)to[out=60,in=-60] (m3);
\end{tikzpicture}
\quad
\begin{tikzpicture}[xscale=0.6,yscale=0.8,baseline=0.8cm]
\node(1) at (0,0.05) {$\;\;H^{\otimes4}$};
\node(2) at (0,2) {$\;\;H^{\otimes3}$};
\draw[->] (1)--(2);
\end{tikzpicture}
$$
\begin{itemize}
\item put one $H$ between every two consecutive $\bullet$'s 
\item for every consecutive pair $\bullet\bullet$ in $\bullet^m$, if the order of the pair is reversed in the image, apply $S$ (depicted by 
\begin{tikzpicture}[scale=0.3, baseline=1.5pt]
\draw[thick](0,0)--node[anti]{} (0,1);
\end{tikzpicture})
to the corresponding $H$
\item if the distance of the two $\bullet$'s in the image is $k$, apply iterated $\Delta$ to get a morphism $H\to H^{\otimes k}$ (if $k=0$, apply $\epsilon$), and distribute the resulting $H$'s to the $k$ pairs of $\bullet$'s using the symmetry of $\mathcal C$
\item finally multiply $H$'s arriving between consecutive $\bullet$'s in $\bullet^n$. If there is no $H$ arriving to a pair of consecutive $\bullet$'s, use the unit of $H$.
\end{itemize}
The coherence morphisms  of $N_H$ (obtained from Proposition \ref{prop:scosca} and from the commutative algebra structure of $N_H(\bullet^n)$) are as follows: for $m,n>0$ the morphism $N_H(\bullet^m)\otimes N_H(\bullet^n)\to N_H(\bullet^{m+n})$ is
\begin{equation}\label{nervecm}
\id_H^{\otimes(m-1)}\otimes\eta\otimes\id_H^{\otimes(n-1)}\colon H^{\otimes(m-1)}\otimes H^{\otimes(n-1)}\to H^{\otimes(m+n-1)}
\end{equation}
and the remaining coherence morphisms are identities.

If $\mc C=\Set^\text{op}$ and $\otimes=\times$ then $H$ is a group and $N_H$ is its nerve, with the identification $N_H(\bullet^n)=H^{n-1}$ given by \eqref{groupNerveIso}.
The above graphical algorithm can be recovered from this example.


One can recognize nerves of commutative Hopf algebras among all symmetric lax monoidal functors $\ms{Com}\to\C$ by an analogue of the map \eqref{recnerve}. 
We shall say that a braided lax monoidal functor $N\colon \ms{Com}\to\C$ or $N\colon \ms{BrCom}\to\C$ \emph{satisfies the nerve condition} if the morphism $N({\bullet\bullet})^{\otimes (n-1)}\to N(\bullet^{n})$ given by 
\begin{equation}\label{Segal}
\hphantom{(n=5)\quad}
\begin{tikzpicture}[xscale=0.4, yscale=0.5, baseline=0.7cm]
\fill[black!10]%
     (-1.5,3)to[out=-90,in=90](-0.5,0)--
     (1.5,0)..controls +(0.1,1.5) and +(-0.1,1.5)..(2.5,0)--
     (4.5,0)..controls +(0.1,1.5) and +(-0.1,1.5)..(5.5,0)--
     (7.5,0)..controls +(0.1,1.5) and +(-0.1,1.5)..(8.5,0)--
     (10.5,0)to[out=90,in=-90](11.5,3)--cycle;
\draw[black!30,thick]%
     (-1.5,3)to[out=-90,in=90](-0.5,0)
     (1.5,0)..controls +(0.1,1.5) and +(-0.1,1.5)..(2.5,0)
     (4.5,0)..controls +(0.1,1.5) and +(-0.1,1.5)..(5.5,0)
     (7.5,0)..controls +(0.1,1.5) and +(-0.1,1.5)..(8.5,0)
     (10.5,0)to[out=90,in=-90](11.5,3);
\draw (0,0)node[bul]{}to[out=90,in=-90]+(-1,3)node[bul]{};
\draw (1,0)node[bul]{}to[out=90,in=-90]+(1,3)node[bul]{};
\draw (3,0)node[bul]{}to[out=90,in=-90]+(-1,3)node[bul]{};
\draw (4,0)node[bul]{}to[out=90,in=-90]+(1,3)node[bul]{};
\draw (6,0)node[bul]{}to[out=90,in=-90]+(-1,3)node[bul]{};
\draw (7,0)node[bul]{}to[out=90,in=-90]+(1,3)node[bul]{};
\draw (9,0)node[bul]{}to[out=90,in=-90]+(-1,3)node[bul]{};
\draw (10,0)node[bul]{}to[out=90,in=-90]+(1,3)node[bul]{};
\end{tikzpicture}
\quad(n=5)
\end{equation}
is an isomorphism for every $n\geq 2$,
and if the two morphisms
\begin{equation*}
1_\C\xrightarrow{c^N_1} N(1)\xrightarrow{N(\uln \eta)} N(\bullet)
\end{equation*} are also isomorphisms.
We then have the following minor generalization of Proposition \ref{prop:grnerve} (which corresponds to the case $\C=\Set^\mathrm{op}$, ${\otimes}={\times}$).

\begin{prop}\label{prop:comHopf}
Let $\C$ be a SMC. The category of commutative Hopf algebras  in $\C$ is equivalent to the category of symmetric lax monoidal functors
$$N\colon \ms{Com}\to\C$$
satisfying the nerve condition.

 The Hopf algebra corresponding to $N$ is $H=N({\bullet\bullet})$. The commutative algebra structure of $H$ comes from the commutative algebra ${\bullet\bullet}$, i.e.\ the product is 
$$m=\begin{tikzpicture}[xscale=0.6,yscale=0.7,baseline=0.5cm]
\coordinate (diff) at (0.8,0);
\coordinate (dy) at (0,0.3);
\coordinate(A1) at (0,0);
\coordinate(B1) at ($(A1)+(diff)$);
\coordinate(A2) at (2,0);
\coordinate(B2) at ($(A2)+(diff)$);
\coordinate(A3) at ($(B1)+(0,2)$);
\coordinate(B3) at ($(A2)+(0,2)$);
\coordinate(ep) at (0.3,0);
\coordinate(A) at ($(A3)-(dy)$);
\coordinate(B) at ($(B3)-(dy)$);

\fill[black!10] ($(A1)-(ep)$)..controls +(0,1) and +(0,-1)..($(A)-(ep)$)--($(A3)-(ep)$)%
--($(B3)+(ep)$)--($(B)+(ep)$)..controls +(0,-1) and +(0,1)..($(B2)+(ep)$)%
--($(A2)-(ep)$)..controls +(0,.5) and +(0,.5)..($(B1)+(ep)$);

\draw[black!30,thick] ($(A1)-(ep)$)..controls +(0,1) and +(0,-1)..($(A)-(ep)$)--($(A3)-(ep)$)%
($(B3)+(ep)$)--($(B)+(ep)$)..controls +(0,-1) and +(0,1)..($(B2)+(ep)$)%
($(A2)-(ep)$)..controls +(0,.5) and +(0,.5)..($(B1)+(ep)$);

\node(A1) at (A1)[bul]{};
\node(B1) at (B1) [bul]{};
\node(A2) at (A2) [bul]{};
\node(B2) at (B2) [bul]{};
\node(A3) at (A3) [bul]{};
\node(B3) at (B3) [bul]{};

\draw (A2)..controls +(0,1) and +(0,-1)..(A);
\draw (B1)..controls +(0,1) and +(0,-1)..(B);
\draw (A1)..controls +(0,1) and +(0,-1)..(A);
\draw (B2)..controls +(0,1) and +(0,-1)..(B);
\draw (B)--(B3);
\draw(A)--(A3);
\end{tikzpicture}$$
 and the unit is the composition $1_\C\to N(1)\to N({\bullet\bullet})$. The coproduct, the counit, and the antipode are given by 
$$
\Delta=\begin{tikzpicture}[xscale=0.3,yscale=0.6,baseline=0.5cm]
\coordinate (bor) at (0.8,0);
\coordinate(1) at (-1,0);
\coordinate(2) at (1,0);
\coordinate(a) at (-2,2);
\coordinate(b) at (0,2);
\coordinate(c) at (2,2);
\fill[black!10] ($(1)-(bor)$)to[out=90,in=-90]($(a)-(bor)$)--($(c)+(bor)$)to[out=-90,in=90]($(2)+(bor)$);
\draw[black!30,thick]($(1)-(bor)$)to[out=90,in=-90]($(a)-(bor)$) ($(c)+(bor)$)to[out=-90,in=90]($(2)+(bor)$);
\node(1) at (1)[bul]{};
\node(2) at (2)[bul]{};
\node(a) at (a)[bul]{};
\node(b) at (b)[bul]{};
\node(c) at (c)[bul]{};
\draw (1)to[out=90,in=-90](a) (2)to[out=90,in=-90](c);
\end{tikzpicture}\qquad
\epsilon=\begin{tikzpicture}[scale=0.6,baseline=0.5cm]
\fill[black!10] (-0.5,0)to[out=90,in=-90](0,2)--(1,2)to[out=-90,in=90](1.5,0);
\draw[black!30,thick](-0.5,0)to[out=90,in=-90](0,2)  (1,2)to[out=-90,in=90](1.5,0);
\node(1) [bul] at (0,0){};
\node(2) [bul] at (1,0){};
\node(3) [bul] at (0.5,2){};
\draw (1) to[out=90,in=-90](3) (2) to[out=90,in=-90](3);

\end{tikzpicture}\qquad
S=\begin{tikzpicture}[scale=0.6,baseline=0.5cm]
\fill[black!10](-0.3,0)--(-0.3,2)--(1.3,2)--(1.3,0);
\draw[black!30,thick](-0.3,0)--(-0.3,2) (1.3,2)--(1.3,0);
\node(1) at (0,0)[bul]{};
\node(2) at (1,0)[bul]{};
\node(a) at (0,2)[bul]{};
\node(b) at (1,2)[bul]{};
\draw (2)..controls +(0,1) and +(0,-1)..(a);
\draw (1)..controls +(0,1) and +(0,-1)..(b);
\end{tikzpicture}
$$
where we implicitly use the isomorphisms $N(\bullet^3)\cong N({\bullet\bullet})^{\otimes 2}$ and $N(\bullet)\cong 1_\C$ given by the nerve condition. 
\end{prop}
The proof is the same as for Proposition \ref{prop:grnerve} and we leave its details to the reader.

\begin{rem}
Proposition \ref{prop:comHopf} and the construction of $N_H$ out of $H$ have the following more familiar ``non-symmetric'' version, which uses coaugmented coalgebras in place of commutative Hopf algebras. 

Let $\Delta_+\subset\mathsf{Com}$ denote the subcategory with all the objects but only with the morphisms with non-intersecting strands (the strands are allowed to collide at their endpoints). It can be seen as the augmented simplex category, with $[n]=\bullet^{n+1}$, or equivalently as the free strict monoidal category generated by a monoid $\bullet$.

If $\C$ is a monoidal category, we then have an equivalence between lax monoidal functors $N\colon\Delta_+\to\C$ satisfying the nerve condition, and coaugmented coalgebras in $\C$. If $N$ is given, the operations $\Delta,\epsilon,\eta$ on the corresponding coaugmented coalgebra $C=N(\bullet\bullet)$ are as in Proposition \ref{prop:comHopf}. If $C$ is given then the corresponding $N$ is the bar construction of the coaugmented coalgebra $C$.
\end{rem}

\subsection{Nerves of braided Hopf algebras}
Proposition \ref{prop:comHopf} has a straightforward generalization to the world of noncommutative Hopf algebras (announced in \cite[Remark 4]{ja}; a similar result appeared in \cite{KS}): 

\begin{thm}\label{thm:brnerve}
Let $\C$ be a BMC. The category of Hopf algebras with invertible antipodes in $\C$ is equivalent to the category of braided lax monoidal functors
$$N\colon \ms{BrCom}\to\C$$
satisfying the nerve condition.

The Hopf algebra corresponding to $N$ is $H=N({\bullet\bullet})$. The algebra structure of $H$ comes from the algebra structure of ${\bullet\bullet}$ (which is the tensor product of two algebras~$\bullet$), i.e. the product of $H$ is
$$m=\begin{tikzpicture}[xscale=0.6,yscale=0.7,baseline=0.5cm]
\coordinate (diff) at (0.8,0);
\coordinate (dy) at (0,0.3);
\coordinate(A1) at (0,0);
\coordinate(B1) at ($(A1)+(diff)$);
\coordinate(A2) at (2,0);
\coordinate(B2) at ($(A2)+(diff)$);
\coordinate(A3) at ($(B1)+(0,2)$);
\coordinate(B3) at ($(A2)+(0,2)$);
\coordinate(ep) at (0.3,0);
\coordinate(A) at ($(A3)-(dy)$);
\coordinate(B) at ($(B3)-(dy)$);

\fill[black!10] ($(A1)-(ep)$)..controls +(0,1) and +(0,-1)..($(A)-(ep)$)--($(A3)-(ep)$)%
--($(B3)+(ep)$)--($(B)+(ep)$)..controls +(0,-1) and +(0,1)..($(B2)+(ep)$)%
--($(A2)-(ep)$)..controls +(0,.5) and +(0,.5)..($(B1)+(ep)$);

\draw[black!30,thick] ($(A1)-(ep)$)..controls +(0,1) and +(0,-1)..($(A)-(ep)$)--($(A3)-(ep)$)%
($(B3)+(ep)$)--($(B)+(ep)$)..controls +(0,-1) and +(0,1)..($(B2)+(ep)$)%
($(A2)-(ep)$)..controls +(0,.5) and +(0,.5)..($(B1)+(ep)$);

\node(A1) at (A1)[bul]{};
\node(B1) at (B1) [bul]{};
\node(A2) at (A2) [bul]{};
\node(B2) at (B2) [bul]{};
\node(A3) at (A3) [bul]{};
\node(B3) at (B3) [bul]{};

\draw[line width=1ex,black!10]  (A2)..controls +(0,1) and +(0,-1)..(A);
\draw (A2)..controls +(0,1) and +(0,-1)..(A);
\draw[line width=1ex,black!10] (B1)..controls +(0,1) and +(0,-1)..(B);
\draw (B1)..controls +(0,1) and +(0,-1)..(B);
\draw (A1)..controls +(0,1) and +(0,-1)..(A);
\draw (B2)..controls +(0,1) and +(0,-1)..(B);
\draw (B)--(B3);
\draw(A)--(A3);
\end{tikzpicture}
$$
and the unit is the composition $1_\C\to N(1)\to N({\bullet\bullet})$. The coproduct, the counit, and the antipode are given by
$$
\Delta=\begin{tikzpicture}[xscale=0.3,yscale=0.6,baseline=0.5cm]
\coordinate (bor) at (0.8,0);
\coordinate(1) at (-1,0);
\coordinate(2) at (1,0);
\coordinate(a) at (-2,2);
\coordinate(b) at (0,2);
\coordinate(c) at (2,2);
\fill[black!10] ($(1)-(bor)$)to[out=90,in=-90]($(a)-(bor)$)--($(c)+(bor)$)to[out=-90,in=90]($(2)+(bor)$);
\draw[black!30,thick]($(1)-(bor)$)to[out=90,in=-90]($(a)-(bor)$) ($(c)+(bor)$)to[out=-90,in=90]($(2)+(bor)$);
\node(1) at (1)[bul]{};
\node(2) at (2)[bul]{};
\node(a) at (a)[bul]{};
\node(b) at (b)[bul]{};
\node(c) at (c)[bul]{};
\draw (1)to[out=90,in=-90](a) (2)to[out=90,in=-90](c);
\end{tikzpicture}\qquad
\epsilon=\begin{tikzpicture}[scale=0.6,baseline=0.5cm]
\fill[black!10] (-0.5,0)to[out=90,in=-90](0,2)--(1,2)to[out=-90,in=90](1.5,0);
\draw[black!30,thick](-0.5,0)to[out=90,in=-90](0,2)  (1,2)to[out=-90,in=90](1.5,0);
\node(1) [bul] at (0,0){};
\node(2) [bul] at (1,0){};
\node(3) [bul] at (0.5,2){};
\draw (1) to[out=90,in=-90](3) (2) to[out=90,in=-90](3);
\end{tikzpicture}\qquad
S=\begin{tikzpicture}[scale=0.6,baseline=0.5cm]
\fill[black!10](-0.3,0)--(-0.3,2)--(1.3,2)--(1.3,0);
\draw[black!30,thick](-0.3,0)--(-0.3,2) (1.3,2)--(1.3,0);
\node(1) at (0,0)[bul]{};
\node(2) at (1,0)[bul]{};
\node(a) at (0,2)[bul]{};
\node(b) at (1,2)[bul]{};
\draw (2)..controls +(0,1) and +(0,-1)..(a);
\draw[line width=1ex,black!10] (1)..controls +(0,1) and +(0,-1)..(b);
\draw (1)..controls +(0,1) and +(0,-1)..(b);
\end{tikzpicture}
$$
where we implicitly use the isomorphisms $N(\bullet^3)\cong N({\bullet\bullet})^{\otimes 2}$ and $N(\bullet)\cong 1_\C$ given by the nerve condition. 
\end{thm}

The proof can be found in \S\ref{sec:pfbr}. Its main part is the construction of a suitable $N$ out of $H$, which is a braided version of the above-given construction $H\mapsto N_H$ for commutative Hopf algebras.

\begin{rem}
	The nerve condition can be seen as strong monoidality of the functor $N\circ s\colon \mathsf{BrCom} \to \mathcal C$, where $s\colon \mathsf{BrCom} \to \mathsf{BrCom}, \; \bullet^n\mapsto \bullet^{n+1}$ is described in \cite[Rem.~1]{ja}. However, $N\circ s$ will usually not be braided; to get a braided strong monoidal functor, one needs to replace $\mathcal C$ as in \cite[Rem.~5]{ja}.
\end{rem}

\section{Infinitesimal braidings}

\subsection{Infinitesimally braided categories}

If $R$ is a commutative ring, an \emph{$R$-linear category} is a category enriched over $R$-modules, i.e.\ a category such that $\Hom(X,Y)$ is an $R$-module for any two objects $X$ and $Y$ and such that the composition of morphisms is $R$-bilinear.

An \emph{$R$-linear monoidal category}  is a category which is both $R$-linear and monoidal and such that the tensor product of morphisms is $R$-bilinear.

If $\C$ is an $R$-linear category, let $\C_\epsilon$ be the $R[\epsilon]/(\epsilon^2)$-linear category obtained by extension of scalars, i.e.\ $\C_\epsilon$ has the same objects as $\C$ and 
$$\Hom_{\C_\epsilon}(X,Y)=\Hom_\C(X,Y)[\epsilon]/(\epsilon^2).$$
An $R$-\emph{infinitesimally braided monoidal category} \cite[Sec.~4.]{Cartier} ($R$-iBMC) is a symmetric $R$-linear monoidal category $\C$ with a natural transformation $t_{X,Y}\colon X\otimes Y\to X\otimes Y$ called \emph{infinitesimal braiding} such that $\C_\epsilon$ with the modified braiding 
$$\beta_{X,Y}\vcentcolon =\sigma_{X,Y}\circ(1+\epsilon t_{X,Y})$$
(where $\sigma_{X,Y}\colon X\otimes Y\to Y\otimes X$ is the symmetry in $\C$)
 is a braided monoidal category. Moreover, $t_{X,Y}$ is required to satisfy the symmetry condition
$$t_{Y,X}=\sigma_{X,Y}\circ t_{X,Y}\circ\sigma_{Y,X}.$$
The morphism $t_{X,Y}$ will be graphically represented by a horizontal chord
$$
\begin{tikzpicture}[xscale=0.7,yscale=0.5]
\draw[thick] (0,0)node[below]{$X$}--node[rbul](a){}(0,2)node[above]{$X$} (1,0)node[below]{$Y$}--node[rbul](b){}(1,2)node[above]{$Y$};
\draw[red](a)--(b);
\end{tikzpicture}
$$

An \emph{i-braided lax/strong monoidal functor} $\C_1\to\C_2$ between two $R$-iBMCs is an $R$-linear symmetric lax/strong monoidal functor such that
$$F(t^{\C_1}_{X,Y})\circ c^F_{X,Y}=c^F_{X,Y}\circ t^{\C_2}_{F(X),F(Y)}.$$

Let us notice that any $R$-linear SMC becomes an $R$-iBMC if we set $t_{X,Y}=0$.

\begin{example}[{\cite{drass}}]\label{ex:gandt}
If $\g$ is a Lie algebra over $R$ and $t\in \g\otimes\g$ is symmetric and $\g$-invariant then the symmetric monoidal category
$U\g$-mod is $R$-i-braided via 
$$t_{X,Y}=\rho_X\otimes\rho_Y (t)\in\on{End}(X\otimes Y).$$
\end{example}

\subsection{Drinfeld associators}
An iBMC is  a SMC with a first order deformation of the symmetry to a braiding. A natural question is whether one can extend this first order deformation to a formal deformation. 

If $\C$ is an $R$-linear category, let $\C_\hbar$ be the $R[\![\hbar]\!]$-linear category with $\on{Ob}(\C_\hbar)=\on{Ob}(\C)$ and with 
$$\Hom_{\C_\hbar}(X,Y)=\Hom_\C(X,Y)[\![\hbar]\!].$$
Likewise, if $F\colon \C\to\C'$ is an $R$-linear functor, let $F_\hbar\colon \C_\hbar\to\C'_\hbar$ be its (continuous) $R[\![\hbar]\!]$-linear extension, i.e.
$$ F_\hbar\bigl(\sum f_n \hbar^n\bigr) = \sum F(f_n) \hbar^n, \quad \quad f_n \in \Hom_\C(X, Y). $$

\begin{thd}[{Drinfeld \cite[Thm.~A$^{\prime\prime}$]{drass}}]\label{thm:DrCat} \label{thd:drinfeld}There is a non-commutative power series 
$$\Phi\in\Q\la\!\la x,y\ra\!\ra$$
 (a \emph{Drinfeld associator}), $\Phi=1+[x,y]/24+\dots$,  with the following property. If $\Q\subset R$ and if $\C$ is a $R$-iBMC then $\C_\hbar$, with the new braiding $\beta^\text{new}_{X,Y}=\sigma_{X,Y}\circ \exp(\hbar\,t_{X,Y}/2)$ and the new associativity constraint 
$$\gamma^\text{new}_{X,Y,Z}=\gamma_{X,Y,Z}\circ\Phi\bigl(\hbar\,t_{X,Y}\otimes\on{id}_Y,\hbar\,\gamma_{X,Y,Z}^{-1}\circ(\on{id}_X\otimes t_{Y,Z})\circ\gamma_{X,Y,Z}\bigr),$$
(where $\gamma_{X,Y,Z}\colon(X\otimes Y)\otimes Z\to X\otimes(Y\otimes Z)$ is the associativity constraint of $\C$) is a BMC which we shall denote by $\C_\hbar^\Phi$. 

If $F\colon \C_1\to\C_2$ is an $R$-i-braided lax/strong monoidal functor then $$F_\hbar\colon (\C_1)_\hbar^\Phi\to(\C_2)_\hbar^\Phi,$$ with the same coherence morphisms, is braided lax/strong monoidal.
\end{thd}

Notice that
$$\beta^\mathit{new}_{X,Y}=\sigma_{X,Y}\circ(1+\tfrac\hbar2\,t_{X,Y})+O(\hbar^2),\quad\gamma^\mathit{new}_{X,Y,Z}=\gamma_{X,Y,Z}+O(\hbar^2).$$
Also notice that if $t_{X,Y}=0$ for all objects $X,Y\in\C$ then $\C_\hbar^\Phi=\C_\hbar$ as a symmetric monoidal category, as then $\beta^\mathit{new}_{X,Y}=\sigma_{X,Y}$ and $\gamma^\mathit{new}_{X,Y,Z}=\gamma_{X,Y,Z}$.

\subsection{Poisson algebras in infinitesimally braided categories} \label{ssec:PalgiBMC}

An (infinitesimally braided)  \emph{Poisson algebra} in an iBMC $\C$ is a commutative algebra $A\in\C$  together with a biderivation 
\[p\colon A\otimes A\to A\]
 (the ``Poisson bracket'')  satisfying the modified skew-symmetry 
\begin{equation}\label{skewmod}
p+p\circ\sigma_{A,A}=m\circ t_{A,A}
\end{equation}
(where $m\colon A\otimes A\to A$ is the product) and the modified Jacobi identity
\begin{equation}\label{Jacobi}
\begin{tikzpicture}[scale=0.5,baseline]
\node[bbul] (A) at (0,1) {};
\node[bbul] (B) at (0.5,0) {};
\draw[thick] (-1,-1)--(A)--(B)--(0,-1) (B)--(1,-1) (A)--(0,2);
\end{tikzpicture}
=
\begin{tikzpicture}[scale=0.5,baseline]
\node[bbul] (A) at (0,1) {};
\node[bbul] (B) at (-0.5,0) {};
\draw[thick] (1,-1)--(A)--(B)--(0,-1) (B)--(-1,-1) (A)--(0,2);
\end{tikzpicture}
+
\begin{tikzpicture}[scale=0.5,baseline]
\node[bbul] (A) at (0,1) {};
\node[bbul] (B) at (0.5,0) {};
\draw[thick] (-1,-1)--(B)--(A)--(0,-1) (B)--(1,-1) (A)--(0,2);
\end{tikzpicture}
-
\begin{tikzpicture}[scale=0.5,baseline]
\coordinate (A) at (0,1);
\node[bbul] (B) at (0.5,0) {};
\draw[thick] (-1,-1)--node(a2)[pos=0.3, rbul]{}(B)--(A)--node(a1)[pos=0.85, rbul]{}(0,-1) (B)--(1,-1) (A)--(0,2);
\draw[red](a1)--(a2);
\end{tikzpicture}
\end{equation}
(using the notation
$p=\begin{tikzpicture}[scale=0.15, baseline=-3pt]
\node[bbul](a) at (0,0) {};
\draw[thick](0.7,-1)--(a)--(-0.7,-1) (0,1)--(a);
\end{tikzpicture}$ and
$m=\begin{tikzpicture}[scale=0.15, baseline=-3pt]
\coordinate(a) at (0,0);
\draw[thick](0.7,-1)--(a)--(-0.7,-1) (0,1)--(a);
\end{tikzpicture}$).

A commutative algebra $A\in\C$ is \emph{strongly commutative} if $m\circ t_{A,A}=0$. Then $A$ with $p=0$ is a Poisson algebra.

If $A\in\C$ is a Poisson algebra  and if $F\colon \C\to\C'$ is an i-braided lax monoidal functor then $F(A)\in\C'$ is a Poisson algebra too. 

If $A_1,A_2\in\C$ are Poisson algebras then so is $A_1\otimes A_2$, with the Poisson bracket 
\begin{equation}\label{pfusion}
\begin{tikzpicture}[scale=0.5,baseline]
\node[bbul](A) at(0,0.5){};
\coordinate(B) at(0.5,0.5);
\draw[thick] (-1,-1)--(A)--(1,-1) (-0.5,-1)--(B)--(1.5,-1) (A)--(0,1.5) (B)--(0.5,1.5);
\end{tikzpicture}
+
\begin{tikzpicture}[scale=0.5,baseline]
\coordinate(A) at(0,0.5);
\node[bbul](B) at(0.5,0.5){};
\draw[thick] (-1,-1)--(A)--(1,-1) (-0.5,-1)--(B)--(1.5,-1) (A)--(0,1.5) (B)--(0.5,1.5);
\end{tikzpicture}
+
\begin{tikzpicture}[scale=0.5,baseline]
\coordinate(A) at(0,0.5);
\coordinate(B) at(0.5,0.5){};
\draw[thick] (-1,-1)--(A)--node(a1)[rbul]{}(1,-1) (-0.5,-1)--node(a2)[rbul]{}(B)--(1.5,-1) (A)--(0,1.5) (B)--(0.5,1.5);
\draw[red] (a1)--(a2);
\end{tikzpicture}
\end{equation}
 Under this tensor product, Poisson algebras in $\C$ form a monoidal category.
 
If $\C$ is an $R$-linear SMC with $t_{X,Y}=0$ (for all objects $X,Y$) then these notions reduce to the usual definition of Poisson algebras in a linear SMC, and to the usual definition of their tensor products.

\begin{example}
If $\C=U\g$-mod is as in Example \ref{ex:gandt} and if $1/2\in R$ then Poisson algebras in $\C$ are the same as $\g$-quasi-Poisson algebras defined in \cite[Def.~2.1]{AKM}. Namely if we split a Poisson bracket $p$ on $A\in\C$ to its anti-symmetric and symmetric parts
$$p=\{,\}+\frac12\rho_A\otimes\rho_A(t)$$
then $\{,\}$ is a $\g$-invariant biderivation of the commutative algebra $A$ satisfying 
$$\{a,\{b,c\}\}+c.p.=-\frac{1}{4}\phi_A(a\otimes b\otimes c)$$
where $\phi_A\colon A^{\otimes 3}\to A$ is given by $\phi_A=m^{(3)}\circ\rho_A^{\otimes 3}(\phi)$, where  $m^{(3)}\colon A^{\otimes 3}\to A$ is the product of 3 elements and $\phi=[t\otimes1,1\otimes t]\in {\textstyle\bigwedge}^3\g\subset(U\g)^{\otimes3}$.

Tensor product of Poisson algebras then corresponds to the fusion product from \textit{op.\ cit.}
\end{example}

The definition of Poisson algebras and of their tensor product come from the following fact. Suppose  that $\Q\subset R$ and let $A$ be an algebra in $\C_\hbar^\Phi$ with a product 
$$m_{}=\sum_{n=0}^\infty\hbar^n m_n\qquad \bigl(m_n\in\Hom_\C(A\otimes A,A)\bigr)$$
such that $m_0$ is commutative. If we define $p\in\Hom_\C(A\otimes A,A)$ via the braided commutator
\begin{equation}\label{bmp}
m_{}-m_{}\circ\beta_{A,A}^{-1}=\hbar\,p + O(\hbar^2), \text{ i.e. } p=m_1-m_1\circ\sigma_{A,A}+\frac12m_0\circ t_{A,A}
\end{equation}
then $m_0$ and $p$ make $A$ to a Poisson algebra in the iBMC $\C$. We shall then say that $m_{}$ is a \emph{quantization} of the Poisson bracket $p$.

 Moreover, if $A_1$ and $A_2$ are two such algebras in $\C_\hbar^\Phi$ then the resulting Poisson bracket on $A_1\otimes A_2$ is given by \eqref{pfusion}.

\begin{prop}[{\cite[Prop.~2]{ja}}]\label{prop:strong-comm}
 If $\C$ is an $R$-iBMC and $A\in\C$ is a strongly commutative algebra then $A$ remains, with the same product (i.e.\ with $m_{}=m_0$) and unit, a commutative algebra in $\C_\hbar^\Phi$.
\end{prop}
\begin{proof}
Since $m_0\circ t_{A,A}=0$ and $m_0\circ\sigma_{A,A}=m_0$ (strong commutativity), we get $m_0\circ\beta_{A,A}=m_0$, i.e.\ the product $m_0$ is commutative in $\C_\hbar^\Phi$. Similarly, $m_0\circ t_{A,A}=0$ and the associativity of $m_0$ in $\C$ imply its associativity in $\C_\hbar^\Phi$.
\end{proof}

\subsection{Infinitesimally braided props}
An \emph{$R$-infinitesimally braided prop} (or $R$-i-braided prop) is a strict $R$-iBMC with objects $\bullet^n$, $n\geq0$, with the tensor product of objects $\bullet^m\otimes\bullet^n=\bullet^{m+n}$, and with the unit $1=\bullet^0$.

Let $\mc A_n$ be the Drinfeld-Kohno algebra \cite{drass}
\begin{multline*}
\mc A_n\vcentcolon=R\bigl\la t_{ij}, 1\leq i,j\leq n, i\neq j \mid t_{ij}=t_{ji}, [t_{ij}+t_{ik}, t_{jk}]=0,\\ [t_{ij},t_{kl}]=0\text{ if }i,j,k,l\text{ are all different}\bigr\ra
\end{multline*}
on which $S_n$ acts by permuting the indices. For any $R$-i-braided prop $\mc P$ we have a canonical map of $R$-algebras
$S_n\ltimes\mc A_n\to\on{End}(\bullet^n)$,
 with $S_n$ coming from the symmetric monoidal category structure and $t_{ij}$ being the chord connecting the $i$'th and $j$'th $\bullet$ in $\bullet^n$.

\begin{example}
Let $\ms{iCom}(R)$ be the $R$-i-braided prop of strongly commutative algebras (initial among strict $R$-iBMCs with a chosen strongly commutative algebra.) It is generated by morphisms $\uln m\colon {\bullet\bullet}\to\bullet$ and $\uln \eta\colon 1\to\bullet$ modulo the relations saying that $\uln m$ is strongly commutative and associative and that $\uln \eta$ is its unit. Morphisms are then $R$-linear combinations of ``maps with chords''
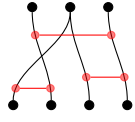
$$
\begin{tikzpicture}[baseline=0.5cm]
\node(A1) at (0,0)[bul]{};
\node(A2) at (.5,0)[bul]{};
\node(A3) at (1,0)[bul]{};
\node(A4) at (1.5,0)[bul]{};
\node(B1) at (.25,1.3)[bul]{};
\node(B2) at (.75,1.3)[bul]{};
\node(B3) at (1.25,1.3)[bul]{};
\draw (A1)..controls +(0,.5) and +(0,-.5)..node(a1)[very near start, rbul]{}(B2);
\draw (A2)..controls +(0,.5) and +(0,-.5)..node(a2)[very near start, rbul]{}node(c1)[near end, rbul]{}(B1);
\draw (A3)..controls +(0,.5) and +(0,-.5)..node(b1)[near start, rbul]{}(B2);
\draw (A4)..controls +(0,.5) and +(0,-.5)..node(b2)[near start, rbul]{}node(c2)[near end, rbul]{}(B3);
\draw[red](a1)--(a2) (b1)--(b2) (c1)--(c2);
\node(X)at(0,-0.2){};
\end{tikzpicture}
$$
modulo the strong commutativity relation 
$$
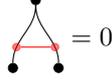
\begin{tikzpicture}[scale=0.6,baseline=0.3cm]
\node(A)[bul] at (0,0) {};
\node(B)[bul] at (1,0) {};
\node(C)[bul] at (0.5,1.5) {};
\draw (A) to[out=90,in=-90] node(a)[near start,rbul]{} (C) (B) to[out=90,in=-90] node(b)[near start,rbul]{}(C);
\draw[red] (a)--(b);

\end{tikzpicture}
=0
$$ and the Drinfeld-Kohno relations (the defining relations of $\mc A_n$).

\end{example}

\section{Poisson Hopf algebras and their quantization}

\subsection{Poisson Hopf algebras in infinitesimally braided categories}

An (infinitesimally braided) \emph{Poisson Hopf algebra} \cite[Sec.~2]{drqg} in an iBMC $\C$ is a commutative Hopf algebra $H\in\C$ together with a Poisson bracket $p\colon H\otimes H\to H$ such that the coproduct $\Delta\colon H\to H\otimes H$ is a morphism of Poisson algebras. In other words, $p$ has to satisfy
\begin{equation}\label{pHA}
\begin{tikzpicture}[scale=0.5,baseline=0.7cm]
\coordinate(in1) at (-1,0);
\coordinate(in2) at (1,0);
\node[bbul](m1) at (0,1){};
\coordinate(m2) at (0,2);
\coordinate(out1) at (-1,3);
\coordinate(out2) at (1,3);
\draw[thick](in1)--(m1)--(in2) (m1)--(m2) (out1)--(m2)--(out2);
\end{tikzpicture}
\ =\ %
\begin{tikzpicture}[scale=0.5,baseline=0.7cm]
\coordinate(in1) at (-1,0);
\coordinate(in2) at (1,0);
\coordinate(m1L) at (-0.7,0.8);
\node[bbul](m2L) at (-0.7,2.2){};
\coordinate(m1R) at (0.7,0.8);
\coordinate(m2R) at (0.7,2.2);
\coordinate(out1) at (-1,3);
\coordinate(out2) at (1,3);
\draw[thick](m1R)--(m2L);
\draw[thick](m1L)--(m2R);
\draw[thick](in1)--(m1L)--(m2L)--(out1) (in2)--(m1R)--(m2R)--(out2);
\end{tikzpicture}
\ +\ %
\begin{tikzpicture}[scale=0.5,baseline=0.7cm]
\coordinate(in1) at (-1,0);
\coordinate(in2) at (1,0);
\coordinate(m1L) at (-0.7,0.8);
\coordinate(m2L) at (-0.7,2.2);
\coordinate(m1R) at (0.7,0.8);
\node[bbul](m2R) at (0.7,2.2){};
\coordinate(out1) at (-1,3);
\coordinate(out2) at (1,3);
\draw[thick](m1R)--(m2L);
\draw[thick](m1L)--(m2R);
\draw[thick](in1)--(m1L)--(m2L)--(out1) (in2)--(m1R)--(m2R)--(out2);
\end{tikzpicture}
\ +\ %
\begin{tikzpicture}[scale=0.5,baseline=0.7cm]
\coordinate(in1) at (-1,0);
\coordinate(in2) at (1,0);
\coordinate(m1L) at (-0.7,0.8);
\coordinate(m2L) at (-0.7,2.2);
\coordinate(m1R) at (0.7,0.8);
\coordinate(m2R) at (0.7,2.2);
\coordinate(out1) at (-1,3);
\coordinate(out2) at (1,3);
\draw[thick](m1R)--node[rbul,pos=0.25](a){}(m2L);
\draw[thick](m1L)--node[rbul,pos=0.25](b){}(m2R);
\draw[thick](in1)--(m1L)--(m2L)--(out1) (in2)--(m1R)--(m2R)--(out2);
\draw[red](a)--(b);
\end{tikzpicture}
\end{equation}

Suppose that $\Q\subset R$ and fix an associator $\Phi$. A \emph{quantization} of a Poisson Hopf algebra 
$$(H,m_0,\Delta_0,\eta,\epsilon,S_0,p)$$
 in $\C$ is a deformation of its product, coproduct, and antipode
$$m_{}=\sum_{n=0}^\infty \hbar^n m_n\qquad \Delta_{}=\sum_{n=0}^\infty \hbar^n \Delta_n\qquad S_{}=\sum_{n=0}^\infty \hbar^n S_n$$
s.t.\ $(H,m_{},\Delta_{},\eta,\epsilon,S_{})$ is a Hopf algebra in the BMC $\C_\hbar^\Phi$ and such that \eqref{bmp} holds.

\begin{example}
An example of Poisson Hopf algebras in nontrivial iBMCs is given by Manin quadruples \cite[Def.~4.1]{SV}, generalizing the standard link between Poisson Hopf algebras, Lie bialgebras and Manin triples \cite[Sec.~3]{drqg}.

A \emph{Manin quadruple} is a quadruple ($\dd,\h,\h^*\!,\g)$, where $\dd$ is a finite-dimensional (or suitably topological) quadratic Lie algebra and $\h,\h^*\!,\g\subset\dd$ are Lie subalgebras such that 
$$\dd=\h\oplus\h^*\oplus\g$$
as a vector space and such that
$$\h^\perp=\h\oplus\g,\quad \h^{*\perp}=\h^*\oplus\g.$$
As an example, $\dd$ can be semisimple, $\p,\p^*\subset\dd$ a pair of opposite parabolic subalgebras,  $\h$ and $\h^*$  their nilpotent radicals, and $\g=\p\cap\p^*$.

The Lie subalgebra $\g\subset\dd$ is  quadratic (i.e.\ the restriction of the quadratic form to $\g$ remains non-degenerate), and thus the category $\C=U\g$-mod is infinitesimally braided.
 $U\h$ is naturally a Poisson Hopf algebra in $\C^\text{op}$. 
See \cite{SV} for details and for the corresponding generalization of Poisson-Lie groups.
\end{example}

\subsection{Nerves of Poisson Hopf algebras}

For Poisson Hopf algebras we have the following analogue of Theorem \ref{thm:brnerve}.

\begin{thm}\label{thm:inerve}
Let $\C$ be a $R$-iBMC. The category of Poisson Hopf algebras in $\C$ is equivalent to the category of i-braided lax monoidal functors
$$N\colon \ms{iCom}(R)\to\C$$
satisfying the nerve condition. The commutative Hopf algebra corresponding to $N$ is $H=N({\bullet\bullet})$, given by Proposition \ref{prop:comHopf} (using the inclusion $\ms{Com}\subset\ms{iCom}(R)$), and the Poisson bracket on $H$ comes from the Poisson bracket on ${\bullet\bullet}=\bullet\otimes\bullet$ (where $\bullet$ has $p=0$), i.e.
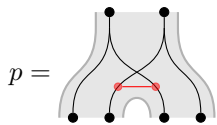
$$p=\begin{tikzpicture}[xscale=0.6,yscale=0.7,baseline=0.5cm]
\coordinate (diff) at (0.8,0);
\coordinate (dy) at (0,0.3);
\coordinate(A1) at (0,0);
\coordinate(B1) at ($(A1)+(diff)$);
\coordinate(A2) at (2,0);
\coordinate(B2) at ($(A2)+(diff)$);
\coordinate(A3) at ($(B1)+(0,2)$);
\coordinate(B3) at ($(A2)+(0,2)$);
\coordinate(ep) at (0.3,0);
\coordinate(A) at ($(A3)-(dy)$);
\coordinate(B) at ($(B3)-(dy)$);

\fill[black!10] ($(A1)-(ep)$)..controls +(0,1) and +(0,-1)..($(A)-(ep)$)--($(A3)-(ep)$)%
--($(B3)+(ep)$)--($(B)+(ep)$)..controls +(0,-1) and +(0,1)..($(B2)+(ep)$)%
--($(A2)-(ep)$)..controls +(0,.5) and +(0,.5)..($(B1)+(ep)$);

\draw[black!30,thick] ($(A1)-(ep)$)..controls +(0,1) and +(0,-1)..($(A)-(ep)$)--($(A3)-(ep)$)%
($(B3)+(ep)$)--($(B)+(ep)$)..controls +(0,-1) and +(0,1)..($(B2)+(ep)$)%
($(A2)-(ep)$)..controls +(0,.5) and +(0,.5)..($(B1)+(ep)$);

\node(A1) at (A1)[bul]{};
\node(B1) at (B1) [bul]{};
\node(A2) at (A2) [bul]{};
\node(B2) at (B2) [bul]{};
\node(A3) at (A3) [bul]{};
\node(B3) at (B3) [bul]{};

\draw (A2)..controls +(0,1) and +(0,-1)..node(c1)[near start, rbul]{}(A);
\draw (B1)..controls +(0,1) and +(0,-1)..node(c2)[near start, rbul]{}(B);
\draw (A1)..controls +(0,1) and +(0,-1)..(A);
\draw (B2)..controls +(0,1) and +(0,-1)..(B);
\draw (B)--(B3);
\draw(A)--(A3);
\draw[red](c1)--(c2);
\end{tikzpicture}$$
\end{thm}

The proof can be found in \S\ref{sec:pfi}.

\subsection{The ``universal quantization functor''}

\begin{prop}
There is a braided strong monoidal functor 
$$U_\Phi\colon \ms{BrCom}\to\ms{iCom}(\Q)_\hbar^\Phi$$
 whose reduction mod $\hbar$ is the projection $\ms{BrCom}\to\ms{Com}$.
\end{prop}
\begin{proof}
	The algebra $\bullet\in\ms{iCom}(\Q)$ is strongly commutative. Thus (Proposition \ref{prop:strong-comm}), $\bullet$ remains, with the same product and unit, a braided commutative algebra in $\ms{iCom}(\Q)_\hbar^\Phi$. Since $\mathsf{BrCom}$ is the prop of braided commutative algebras, the algebra $\bullet$ gives us a braided strong monoidal functor $\ms{BrCom}\to\ms{iCom}(\Q)_\hbar^\Phi$ sending $\bullet$ to $\bullet$. Its reduction mod $\hbar$ has clearly the required property.
\end{proof}
\begin{rem}
The functor $U_\Phi$ is not unique, since it depends on a choice of parenthesizations, as explained in Remark~\ref{rem:parfirst}. Alternatively, one can introduce the parenthesizations to $\ms{BrCom}$ and get a unique $U_\Phi$, see Remark \ref{rem:Pa}.
\end{rem}

\subsection{Quantization of Poisson Hopf algebras}
Poisson Hopf algebras can be quantized as follows.

\begin{thm}\label{thm:quantPHA}
Let $H$ be a Poisson Hopf algebra in an $R$-iBMC $\C$, and let 
$$N\colon \ms{iCom}(R)\to\C$$
be its nerve. Suppose that $\Q\subset R$, and let $\Phi$ be a Drinfeld associator. Then the composed braided lax monoidal functor
$$\ms{BrCom}\xrightarrow{U_\Phi}\ms{iCom}(\Q)_\hbar^\Phi\subset\ms{iCom}(R)_\hbar^\Phi\xrightarrow{N_\hbar}\C_\hbar^\Phi$$
satisfies the nerve condition. The resulting Hopf algebra structure on $H\in\C_\hbar^\Phi$ (cf.\ Theorem \ref{thm:brnerve}) is a quantization of the Poisson Hopf algebra $H\in\C$.
\end{thm}
\begin{proof}
To see that the composed functor $\tilde N$ satisfies the nerve condition we reduce it mod $\hbar$, and get the composition
$$\ms{BrCom}\xrightarrow{\mathrm{projection}}\ms{Com}\xrightarrow{N}\C$$
which satisfies the nerve condition because $N$ does. Since $\tilde N$ is its $\hbar$-deformation, it also satisfies the nerve condition (which says that some morphisms should be invertible).

We need to check that we get a quantization of the Poisson Hopf algebra. If $m_{}$ is the product on $H\in\C_\hbar^\Phi$, we have

$$m_{}-m_{}\circ\beta^{-1}_{H,H}=
\begin{tikzpicture}[xscale=0.6,yscale=0.7,baseline=0.5cm]
\coordinate (diff) at (0.8,0);
\coordinate (dy) at (0,0.3);
\coordinate(A1) at (0,0);
\coordinate(B1) at ($(A1)+(diff)$);
\coordinate(A2) at (2,0);
\coordinate(B2) at ($(A2)+(diff)$);
\coordinate(A3) at ($(B1)+(0,2)$);
\coordinate(B3) at ($(A2)+(0,2)$);
\coordinate(ep) at (0.3,0);
\coordinate(A) at ($(A3)-(dy)$);
\coordinate(B) at ($(B3)-(dy)$);

\fill[black!10] ($(A1)-(ep)$)..controls +(0,1) and +(0,-1)..($(A)-(ep)$)--($(A3)-(ep)$)%
--($(B3)+(ep)$)--($(B)+(ep)$)..controls +(0,-1) and +(0,1)..($(B2)+(ep)$)%
--($(A2)-(ep)$)..controls +(0,.5) and +(0,.5)..($(B1)+(ep)$);

\draw[black!30,thick] ($(A1)-(ep)$)..controls +(0,1) and +(0,-1)..($(A)-(ep)$)--($(A3)-(ep)$)%
($(B3)+(ep)$)--($(B)+(ep)$)..controls +(0,-1) and +(0,1)..($(B2)+(ep)$)%
($(A2)-(ep)$)..controls +(0,.5) and +(0,.5)..($(B1)+(ep)$);

\node(A1) at (A1)[bul]{};
\node(B1) at (B1) [bul]{};
\node(A2) at (A2) [bul]{};
\node(B2) at (B2) [bul]{};
\node(A3) at (A3) [bul]{};
\node(B3) at (B3) [bul]{};

\draw[line width=1ex,black!10]  (A2)..controls +(0,1) and +(0,-1)..(A);
\draw (A2)..controls +(0,1) and +(0,-1)..(A);
\draw[line width=1ex,black!10] (B1)..controls +(0,1) and +(0,-1)..(B);
\draw (B1)..controls +(0,1) and +(0,-1)..(B);
\draw (A1)..controls +(0,1) and +(0,-1)..(A);
\draw (B2)..controls +(0,1) and +(0,-1)..(B);
\draw (B)--(B3);
\draw(A)--(A3);
\end{tikzpicture}
-
\begin{tikzpicture}[xscale=0.6,yscale=0.7,baseline=0.5cm]
\coordinate (diff) at (0.8,0);
\coordinate (dy) at (0,0.3);
\coordinate(A1) at (0,0);
\coordinate(B1) at ($(A1)+(diff)$);
\coordinate(A2) at (2,0);
\coordinate(B2) at ($(A2)+(diff)$);
\coordinate(A3) at ($(B1)+(0,2)$);
\coordinate(B3) at ($(A2)+(0,2)$);
\coordinate(ep) at (0.3,0);
\coordinate(A) at ($(A3)-(dy)$);
\coordinate(B) at ($(B3)-(dy)$);

\fill[black!10] ($(A1)-(ep)$)..controls +(0,1) and +(0,-1)..($(A)-(ep)$)--($(A3)-(ep)$)%
--($(B3)+(ep)$)--($(B)+(ep)$)..controls +(0,-1) and +(0,1)..($(B2)+(ep)$)%
--($(A2)-(ep)$)..controls +(0,.5) and +(0,.5)..($(B1)+(ep)$);

\draw[black!30,thick] ($(A1)-(ep)$)..controls +(0,1) and +(0,-1)..($(A)-(ep)$)--($(A3)-(ep)$)%
($(B3)+(ep)$)--($(B)+(ep)$)..controls +(0,-1) and +(0,1)..($(B2)+(ep)$)%
($(A2)-(ep)$)..controls +(0,.5) and +(0,.5)..($(B1)+(ep)$);

\node(A1) at (A1)[bul]{};
\node(B1) at (B1) [bul]{};
\node(A2) at (A2) [bul]{};
\node(B2) at (B2) [bul]{};
\node(A3) at (A3) [bul]{};
\node(B3) at (B3) [bul]{};

\draw[line width=1ex,black!10] (B1)..controls +(0,1) and +(0,-1)..(B);
\draw (B1)..controls +(0,1) and +(0,-1)..(B);
\draw[line width=1ex,black!10]  (A2)..controls +(0,1) and +(0,-1)..(A);
\draw (A2)..controls +(0,1) and +(0,-1)..(A);
\draw (A1)..controls +(0,1) and +(0,-1)..(A);
\draw (B2)..controls +(0,1) and +(0,-1)..(B);
\draw (B)--(B3);
\draw(A)--(A3);
\end{tikzpicture}
=
\hbar
\begin{tikzpicture}[xscale=0.6,yscale=0.7,baseline=0.5cm]
\coordinate (diff) at (0.8,0);
\coordinate (dy) at (0,0.3);
\coordinate(A1) at (0,0);
\coordinate(B1) at ($(A1)+(diff)$);
\coordinate(A2) at (2,0);
\coordinate(B2) at ($(A2)+(diff)$);
\coordinate(A3) at ($(B1)+(0,2)$);
\coordinate(B3) at ($(A2)+(0,2)$);
\coordinate(ep) at (0.3,0);
\coordinate(A) at ($(A3)-(dy)$);
\coordinate(B) at ($(B3)-(dy)$);

\fill[black!10] ($(A1)-(ep)$)..controls +(0,1) and +(0,-1)..($(A)-(ep)$)--($(A3)-(ep)$)%
--($(B3)+(ep)$)--($(B)+(ep)$)..controls +(0,-1) and +(0,1)..($(B2)+(ep)$)%
--($(A2)-(ep)$)..controls +(0,.5) and +(0,.5)..($(B1)+(ep)$);

\draw[black!30,thick] ($(A1)-(ep)$)..controls +(0,1) and +(0,-1)..($(A)-(ep)$)--($(A3)-(ep)$)%
($(B3)+(ep)$)--($(B)+(ep)$)..controls +(0,-1) and +(0,1)..($(B2)+(ep)$)%
($(A2)-(ep)$)..controls +(0,.5) and +(0,.5)..($(B1)+(ep)$);

\node(A1) at (A1)[bul]{};
\node(B1) at (B1) [bul]{};
\node(A2) at (A2) [bul]{};
\node(B2) at (B2) [bul]{};
\node(A3) at (A3) [bul]{};
\node(B3) at (B3) [bul]{};

\draw (A2)..controls +(0,1) and +(0,-1)..node(c1)[near start, rbul]{}(A);
\draw (B1)..controls +(0,1) and +(0,-1)..node(c2)[near start, rbul]{}(B);
\draw (A1)..controls +(0,1) and +(0,-1)..(A);
\draw (B2)..controls +(0,1) and +(0,-1)..(B);
\draw (B)--(B3);
\draw(A)--(A3);
\draw[red](c1)--(c2);
\end{tikzpicture}
+O(\hbar^2)
$$
and thus indeed $m_{}-m_{}\circ\beta^{-1}_{H,H}=\hbar\,p+O(\hbar^2)$ where $p$ is the Poisson bracket of the Poisson Hopf algebra $H\in\C$.
\end{proof}

If $t_{X,Y}=0$ for all objects $X,Y\in\C$ then the previous theorem gives us a quantization of $H$ to a Hopf algebra in the SMC $\C_\hbar$.  

\begin{rem}\label{rem:Pa}
\newcommand{\Pa}{\ms{Pa}}
To see how the quantization depends on choices it is better to use parenthesized versions of the props $\ms{BrCom}$ and $\ms{iCom}$.

 Following \cite[Sec.~2.1]{DBN}, if $\mc P$ is a (ordinary or braided or $R$-i-braided) prop, let $\Pa\mc P$ be its parenthesized version.  It is a (symmetric or braided or $R$-i-braided) non-strict monoidal category equivalent to $\mc P$: its objects are fully parenthesised words in $\bullet$ (i.e.\  expressions built out of $\bullet$ using a non-associative product),  and its morphisms are the morphisms of $\mc P$ forgetting the parenthesization. A typical morphism in $\Pa\ms{BrCom}$ looks as 
$$
\begin{tikzpicture}[xscale=0.6,yscale=0.8]
\node[bul](a1) at (0,0){}; \node[left] at (a1) {$(\!$};
\node[bul](a2) at (1,0){}; \node[left] at (a2) {$(\!$};
\node[bul](a3) at (2,0){}; \node[right] at (a3) {$\!))$};
\node[bul](a4) at (3,0){}; \node[left] at (a4) {$(\!$};
\node[bul](a5) at (4,0){}; \node[right] at (a5) {$\!)$};

\node[bul](b1) at (0.5,2){};
\node[bul](b2) at (1.5,2){}; \node[left] at (b2) {$((\!$};
\node[bul](b3) at (2.5,2){}; \node[right] at (b3) {$\!)$};
\node[bul](b4) at (3.5,2){}; \node[right] at (b4) {$\!)$};

\draw (a2)to[out=90,in=-90](b1);
\draw[line width=1ex,white](a1)to[out=90,in=-90](b2);
\draw(a1)to[out=90,in=-90](b2);
\draw (a3)to[out=90,in=-90](b2);
\draw (a5)to[out=90,in=-90](b2);
\draw[line width=1ex,white] (a4)to[out=90,in=-90](b4);
\draw (a4)to[out=90,in=-90](b4);
\end{tikzpicture}
$$

Theorems \ref{thm:brnerve} and \ref{thm:inerve} remain true with $\Pa\ms{BrCom}$ and $\Pa\ms{iCom}$ in place of $\ms{BrCom}$ and $\ms{iCom}$. 
 The advantage of this setup is that now we have a \emph{unique} strict braided monoidal functor $\Pa\, U_\Phi\colon\Pa\ms{BrCom}\to\Pa\ms{iCom}^\Phi_\hbar$ given by the commutative algebra structure on $\bullet$ and thus the resulting quantization of Poisson Hopf algebras depends only on the associator $\Phi$ and not on the choice of $U_\Phi$ (which depends on a choice of a parenthesization). A similar argument also gives us an action of the Grothendieck-Teichmueller Lie algebra $\mf{grt}$ on the prop of Poisson Hopf algebras by derivations.

\end{rem}

\subsection{Cocommutative coalgebra enrichment}
Let us now consider categories enriched over the category of cocommutative $R$-coalgebras; we shall call them \emph{$R$-cc enriched categories}. 

In an \emph{$R$-cc enriched BMC} the tensor product of morphisms has to be compatible with the enrichment, and the braiding and the associativity and unit constraints
 are required to be grouplike, 
i.e.\ to be $R$-cc enriched natural transformations.  

An \emph{$R$-cc enriched iBMC} is an $R$-cc enriched SMC with an infinitesimal braiding such that $t_{X,Y}$ is primitive for every $X,Y$ (which is equivalent to the deformed braiding $\sigma_{X,Y}\circ(1+\epsilon\,t_{X,Y})$ being grouplike).

\begin{example}
The $R$-i-braided prop $\ms{iCom}(R)$ is $R$-cc enriched, if we demand its generating morphisms $\uln m\colon {\bullet\bullet}\to\bullet$ and $\uln\eta\colon 1\to\bullet$ to be grouplike.
\end{example}

\begin{example}
Let $\ms{iPoissHopf}(R)$ be the $R$-i-braided prop of Poisson Hopf algebras. It is naturally $R$-cc enriched, with all the generating morphisms (product, coproduct, unit, counit, antipode) being grouplike, except for the Poisson bracket $\uln p$ which satisfies $\Delta\uln p=\uln p\otimes\uln m + \uln m\otimes\uln p$.
\end{example}

\begin{example}
Let $\ms{PoissHopf}(R)$ be the $R$-linear prop of Poisson Hopf algebras. It is naturally $R$-cc enriched, in the same way as $\ms{iPoissHopf}(R)$.
\end{example}

The $R$-cc enrichment of the prop $\ms{PoissHopf}(R)$ is equivalent to the fact that if $H_1$ and $H_2$ are Poisson Hopf algebras in an $R$-linear SMC $\C$, then so is $H_1\otimes H_2$. Let us explain it in a way which works also for braided and i-braided props.

If $\C_1$ and $\C_2$ are $R$-linear categories, let $\C_1\boxtimes_R\C_2$ be the $R$-linear category with 
$$\on{Ob}(\C_1\boxtimes_R\C_2)=\on{Ob}(\C_1)\times\on{Ob}(\C_2),$$
$$\Hom_{\C_1\boxtimes_R\C_2}\bigl((X_1,X_2),(Y_1,Y_2)\bigr)=\Hom_{\C_1}(X_1,Y_1)\otimes_R\Hom_{\C_2}(X_2,Y_2).$$

If $\C_1$ and $\C_2$ are  braided $R$-linear monoidal then so is $\C_1\boxtimes_R\C_2$: $\otimes$ is given componentwise, i.e.\ $(X_1,X_2)\otimes(Y_1,Y_2)=(X_1\otimes Y_1,X_2\otimes Y_2)$ and similarly for morphisms, $1_{\C_1\boxtimes_R\C_2}=(1_{\C_1},1_{\C_2})$, and we have 
$$\gamma^{\C_1\boxtimes_R\C_2}=\gamma^{\C_1}\otimes_R\gamma^{\C_2}\quad\text{and}\quad\beta^{\C_1\boxtimes_R\C_2}=\beta^{\C_1}\otimes_R\beta^{\C_2}$$ 
and likewise for the unit morphisms.
Moreover, if $\C_1$ and $\C_2$ are $R$-iBMCs then so is $\C_1\boxtimes_R\C_2$, with the infinitesimal braiding
$$t^{\C_1\boxtimes_R\C_2}=t^{\C_1}\otimes_R\id+\id\mathbin{\otimes_R} t^{\C_2}.$$

If $\C$ is an $R$-cc enriched BMC or iBMC then the functor 
$$\Delta_\C\colon \C\to\C\boxtimes_R\C,$$
 given on objects by $\Delta_\C(X)=(X,X)$ and on morphisms by the coalgebra coproduct, is braided (or infinitesimally braided) strictly monoidal. 
 
 In particular, if $\mc P$ is an $R$-cc enriched prop (or braided prop, or i-braided prop) and if $\mc P\to\C_{1,2}$ are two $R$-linear symmetric (or braided or i-braided) strong monoidal functors, then by composition with $\Delta_\mc P\colon \mc P\to\mc P\boxtimes_R\mc P$ we get a symmetric (or braided or i-braided) strong monoidal functor $\mc P\to\C_1\boxtimes_R\C_2$. In other words, if $A_1\in\mc C_1$ and $A_2\in\mc P_2$ are $\mc P$-algebras, then so is $(A_1,A_2)\in\C_1\boxtimes_R\C_2$.
 
Finally, if $\C_1=\C_2=\C$ is a SMC, the tensor product functor $\otimes\colon\C\boxtimes_R\C\to\C$ is a symmetric strong monoidal functor, and so $A_1\otimes A_2\in\C$, being the $\otimes$-image of the $\mc P$-algebra $(A_1,A_2)$, is a $\mc P$-algebra too.

%

\medskip
Let us conclude with the compatibility of Drinfeld's construction of the BMC $\C_\hbar^\Phi$ with the product of categories $\boxtimes_R$ and with $R$-cc enrichment.
Drinfeld \cite[Thm.~A$^{\prime\prime}$]{drass} showed that $\Phi$ can be chosen grouplike (wrt.\ the coproduct where $x$ and $y$ are primitive); from now on we shall assume that $\Phi$ is grouplike. This implies that 
$$(\C_1\boxtimes_R \C_2)_\hbar^\Phi=(\C_1)_\hbar^\Phi\boxtimes_{R[\![\hbar]\!]} (\C_2)_\hbar^\Phi\quad\text{(as $R[\![\hbar]\!]$-linear BMCs)}$$
 and also that if $\C$ is an $R$-cc enriched iBMC then $\C_\hbar^\Phi$ is an $R[\![\hbar]\!]$-cc enriched BMC. 
(Here we abuse the notation slightly -  our $R[\![\hbar]\!]$-modules are of the form $M[\![\hbar]\!]$ for some $R$-module $M$, with the tensor product $M[\![\hbar]\!]\otimes_{{R[\![\hbar]\!]}}{M'[\![\hbar]\!]}\vcentcolon=(M\otimes_R M')[\![\hbar]\!]$.)

\subsection{Quantization in terms of props and compatibility with products}

If $\mc D$ an $R$-cc enriched BMC, let ${\mc D^\mathit{gl}}\subset\mc D$ be its sub-BMC with all the objects but only the grouplike morphisms. If $\mc C$ is a BMC we shall say that a braided monoidal functor $F\colon \C\to\mc D$ is \emph{grouplike} if its image and coherence morphisms are in $\mc D^\mathit{gl}$.

\begin{example}
The braided strong monoidal functor $U_\Phi\colon \ms{BrCom}\to\ms{iCom}(\Q)_\hbar^\Phi$ is grouplike. This is because the commutative algebra $\bullet\in\ms{iCom}(\Q)_\hbar^\Phi$ is a commutative algebra in the subcategory $\bigl(\ms{iCom}(\Q)_\hbar^\Phi\bigr){}^\mathit{gl}_{\phantom{x}}\!,$ as $\uln{m}$ and $\uln{\eta}$ are grouplike morphisms.

\end{example}

Let us now recast Theorem \ref{thm:quantPHA} into the language of props. Let $\ms{Hopf}$ be the prop of Hopf algebras and $\ms{BrHopf}$  the braided prop of Hopf algebras.

\begin{thm}
There is a grouplike braided strong monoidal functor
$$Q_\Phi^\mathit{br}\colon \ms{BrHopf}\to\ms{iPoissHopf}(\Q)_\hbar^\Phi, \quad Q_\Phi^\mathit{br}(\bullet^n)=\bullet^n\ (\forall n)$$
such that the resulting Hopf algebra structure on $\bullet\in \ms{iPoissHopf}(\Q)_\hbar^\Phi$  is a quantization of the Poisson Hopf algebra structure on $\bullet\in \ms{iPoissHopf}(\Q)$,
and a grouplike morphism of props
$$Q_\Phi\colon \ms{Hopf}\to\ms{PoissHopf}(\Q)_\hbar$$
with the same property.
\end{thm}
\begin{proof}
 Since $\bullet$ is a Poisson-Hopf algebra in $\ms{iPoissHopf}(\Q)$, by Theorem \ref{thm:quantPHA} we get its quantization in $\ms{iPoissHopf}(\Q)_\hbar^\Phi$, and thus a braided strong monoidal functor
$$Q_\Phi^\mathit{br}\colon \ms{BrHopf}\to\ms{iPoissHopf}(\Q)_\hbar^\Phi, \quad Q_\Phi^\mathit{br}(\bullet^n)=\bullet^n\ (\forall n).$$

To see that $Q_\Phi^\mathit{br}$ is grouplike, let us notice that the nerve functor corresponding to the Poisson Hopf algebra $\bullet\in\ms{iPoissHopf}(\Q)$, 
$$N\colon \ms{iCom}(\Q)\to\ms{iPoissHopf}(\Q),$$
 is $R$-cc enriched i-braided lax monoidal. As a result 
$$N_\hbar\colon \ms{iCom}(\Q)_\hbar^\Phi\to\ms{iPoissHopf}(\Q)_\hbar^\Phi$$
 is $R$-cc enriched braided lax monoidal, and thus its composition with $U_\Phi$ is a group\-like braided lax monoidal functor $\ms{BrCom}\to\ms{iPoissHopf}(\Q)_\hbar^\Phi$. As a result, all the defining operations of the resulting braided Hopf algebra $\bullet\in\ms{iPoissHopf}(\Q)_\hbar^\Phi$ are grouplike, and thus $Q_\Phi^\mathit{br}$ is grouplike.

Applying the same construction and reasoning to $\ms{PoissHopf}(\Q)$ we get a grouplike morphism of props
$Q_\Phi\colon \ms{Hopf}\to\ms{PoissHopf}(\Q)_\hbar$.
\end{proof}

\begin{rem}
As explained in \cite[Sec.~4]{EE}, quantization of Lie bialgebras is equivalent to a suitable morphism of props
$$\ms{Hopf}\to\overline{\ms{PoissHopf}}(\Q)_\hbar$$
where $\overline{\ms{PoissHopf}}(\Q)$ is a certain completion of $\ms{PoissHopf}(\Q)$. As we have shown, such a morphism exists also without a completion.
\end{rem}

The previous theorem puts our quantization of Poisson Hopf algebras to the following form. If $H\in\C$ is a Poisson Hopf algebra in an $R$-iBMC with $\Q\subset R$, i.e.\ if we have a $\Q$-i-braided strong monoidal functor $F\colon \ms{iPoissHopf}(\Q)\to\C$ with $F(\bullet)=H$, then we compose $F_\hbar$ with $Q^\textit{br}_\Phi$ and thus make $H$ to a Hopf algebra in $\C_\hbar^\Phi$. If the infinitesimal braiding of $\C$ vanishes then we can work with $\ms{PoissHopf}$ and $Q_\Phi$ in place of $\ms{iPossHopf}$ and $Q^\textit{br}_\Phi$.

If $H\in\C$ is a Poisson Hopf algebra, let $H_\hbar^\Phi\in\C_\hbar^\Phi$ denote the same object $H$ with its new Hopf algebra structure.
The fact that $Q^\textit{br}_\Phi$ and $Q_\Phi$ are grouplike gives us the following.  

\begin{cor}
Suppose that $H_1\in\C_1$ and $H_2\in\C_2$ are Poisson Hopf algebras. Then also $(H_1,H_2)\in\C_1\boxtimes_R\C_2$ is a Poisson Hopf algebra, and we have the equality of Hopf algebras
$$(H_1{}_\hbar^\Phi,H_2{}_\hbar^\Phi)=(H_1,H_2)_\hbar^\Phi \ \in\ \C_1{}_\hbar^\Phi\boxtimes_{R[\![\hbar]\!]}\C_2{}_\hbar^\Phi=(\C_1\boxtimes_R\C_2)_\hbar^\Phi.$$

If $\C$ has vanishing infinitesimal braiding, and thus $\C_\hbar^\Phi=\C_\hbar$ is a  SMC, and if $H_{1,2}\in\C$ are Poisson Hopf algebras, then also
$$(H_1\otimes H_2)_\hbar^\Phi=H_1{}_\hbar^\Phi\otimes H_2{}_\hbar^\Phi.$$
\end{cor}

\section{Proofs of the nerve theorems}\label{sec:proofs}

\subsection{Nerves of braided Hopf algebras (proof of Theorem \ref{thm:brnerve})}\label{sec:pfbr}

We can suppose that $\C$ is strict monoidal. We shall prove the theorem in the following form: there is an \emph{isomorphism} of categories between Hopf algebras with invertible antipodes in $\C$ and braided lax monoidal functors $N\colon\ms{BrCom}\to\C$ which satisfy the \emph{strict nerve condition}: the isomorphisms in the nerve condition are required to be identities.

Let us construct $N=N_H\colon\ms{BrCom}\to\C$ out of a Hopf algebra $H$. By the strict nerve condition we have
$$N_H(\bullet^n)=H^{\otimes(n-1)}.$$
The strict nerve condition also implies that the coherence morphisms $$N_H(\bullet^m)\otimes N_H(\bullet^n)\to N_H(\bullet^{m+n})$$ for $m,n>0$ are 
$$\id_H^{\otimes(m-1)}\otimes\eta\otimes\id_H^{\otimes(n-1)}\colon H^{\otimes(m-1)}\otimes H^{\otimes(n-1)}\to H^{\otimes(m+n-1)}$$
where $\eta\colon 1_\C\to H$ is the unit of $H$, as explained for the case $m=2$, $n=3$:
$$
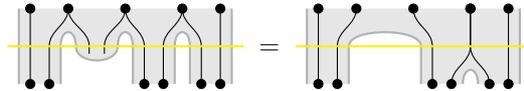
\begin{tikzpicture}[xscale=0.5,yscale=0.5,baseline=-0.1cm]
\coordinate (bigdiff) at (1.5, 0);
\coordinate (S1) at (0, 0);
\coordinate (S2) at ($(S1)+(bigdiff)$);
\coordinate (S3) at ($(S1)+2*(bigdiff)$);
\coordinate (S4) at ($(S1)+3*(bigdiff)$);
\coordinate (xshift) at (0.25, 0);
\coordinate(ep) at (0.3,0);

\coordinate(B2) at ($0.5*(S1)+0.5*(S2)+(0, 1)$);
\coordinate(B3) at ($0.5*(S2)+0.5*(S3)+(0, 1)$);
\coordinate(B4) at ($0.5*(S3)+0.5*(S4)+(0, 1)$);

\fill[black!10]%
($(S1)+(0,1)-(xshift)-(ep)$)--($(S1)-(0,1)-(xshift)-(ep)$)--%
($(S1)-(0,1)+(xshift)+(ep)$)--($(S1)+(xshift)+(ep)$)%
..controls+(0,0.5) and +(0, 0.5)..($(S2)-(xshift)-(ep)$)%
..controls+(0,-0.5) and +(0, -0.5)..($(S2)+(xshift)+(ep)$)%
..controls+(0,0.5) and +(0, 0.5)..($(S3)-(xshift)-(ep)$)--($(S3)-(xshift)-(ep)-(0,1)$)--%
($(S3)-(0,1)+(xshift)+(ep)$)--($(S3)+(xshift)+(ep)$)%
..controls+(0,0.5) and +(0, 0.5)..($(S4)-(xshift)-(ep)$)%
--($(S4)-(0,1)-(xshift)-(ep)$)--%
($(S4)-(0,1)+(xshift)+(ep)$)--($(S4)+(0,1)+(xshift)+(ep)$);

\draw[black!30,thick]%
($(S1)-(0,1)+(xshift)+(ep)$)--($(S1)+(xshift)+(ep)$)%
..controls+(0,0.5) and +(0, 0.5)..($(S2)-(xshift)-(ep)$)%
..controls+(0,-0.5) and +(0, -0.5)..($(S2)+(xshift)+(ep)$)%
..controls+(0,0.5) and +(0, 0.5)..($(S3)-(xshift)-(ep)$)--($(S3)-(xshift)-(ep)-(0,1)$)%
($(S3)-(0,1)+(xshift)+(ep)$)--($(S3)+(xshift)+(ep)$)%
..controls+(0,0.5) and +(0, 0.5)..($(S4)-(xshift)-(ep)$)%
--($(S4)-(0,1)-(xshift)-(ep)$)%
($(S1)-(0,1)-(xshift)-(ep)$)--($(S1)+(0,1)-(xshift)-(ep)$)%
($(S4)-(0,1)+(xshift)+(ep)$)--($(S4)+(0,1)+(xshift)+(ep)$);

\draw ($(S1)-(0,1)-(xshift)$)node[bul]{}--($(S1)+(0,1)-(xshift)$)node[bul]{};
\draw ($(S1)-(0,1)+(xshift)$)node[bul]{}--($(S1)+(0,0)+(xshift)$)to[out=90, in=-90](B2)node[bul]{};
\coordinate (sxshift) at ($0.8*(xshift)$);
\draw ($(S2)-(0,0.2)-(sxshift)$)--($(S2)+(0,0)-(sxshift)$)to[out=90, in=-90](B2);
\draw ($(S2)-(0,0.2)+(sxshift)$)--($(S2)+(0,0)+(sxshift)$)to[out=90, in=-90](B3);
\draw ($(S3)-(0,1)-(xshift)$)node[bul]{}--($(S3)+(0,0)-(xshift)$)to[out=90, in=-90](B3)node[bul]{};
\draw ($(S3)-(0,1)+(xshift)$)node[bul]{}--($(S3)+(0,0)+(xshift)$)to[out=90, in=-90](B4)node[bul]{};
\draw ($(S4)-(0,1)-(xshift)$)node[bul]{}--($(S4)+(0,0)-(xshift)$)to[out=90, in=-90](B4);
\draw ($(S4)-(0,1)+(xshift)$)node[bul]{}--($(S4)+(0,1)+(xshift)$)node[bul]{};

\draw[yellow, thick] ($(S1)-(xshift)-2*(ep)$)--($(S4)+(xshift)+2*(ep)$);
\end{tikzpicture}
\;=\;
\begin{tikzpicture}[xscale=0.5,yscale=0.5,baseline=-0.1cm]
\coordinate (bigdiff) at (1.5, 0);
\coordinate (S1) at (0, 0);
\coordinate (S2) at ($(S1)+(bigdiff)$);
\coordinate (S3) at ($(S1)+2*(bigdiff)$);
\coordinate (S4) at ($(S1)+3*(bigdiff)$);
\coordinate (xshift) at (0.25, 0);
\coordinate(ep) at (0.3,0);

\coordinate(D) at ($0.5*(S3)+0.5*(S4)$);

\fill[black!10]%
($(S1)+(0,1)-(xshift)-(ep)$)--($(S1)-(0,1)-(xshift)-(ep)$)%
--($(S1)-(0,1)+(xshift)+(ep)$)--($(S1)+(xshift)+(ep)$)%
..controls+(0,0.5) and +(0, 0.5)..($(S3)-(xshift)-(ep)$)--($(S3)-(xshift)-(ep)-(0,1)$)%
--($(S3)-(0,1)+(xshift)+(ep)$)%
..controls+(0,0.5) and +(0, 0.5)..%
($(S4)-(0,1)-(xshift)-(ep)$)%
--($(S4)-(0,1)+(xshift)+(ep)$)--($(S4)+(0,1)+(xshift)+(ep)$);

\draw[black!30,thick]%
($(S1)-(0,1)+(xshift)+(ep)$)--($(S1)+(xshift)+(ep)$)%
..controls+(0,0.5) and +(0, 0.5)..($(S3)-(xshift)-(ep)$)--($(S3)-(xshift)-(ep)-(0,1)$)%
($(S3)-(0,1)+(xshift)+(ep)$)%
..controls+(0,0.5) and +(0, 0.5)..%
($(S4)-(0,1)-(xshift)-(ep)$)%
($(S1)-(0,1)-(xshift)-(ep)$)--($(S1)+(0,1)-(xshift)-(ep)$)%
($(S4)-(0,1)+(xshift)+(ep)$)--($(S4)+(0,1)+(xshift)+(ep)$);

\draw ($(S1)-(0,1)-(xshift)$)node[bul]{}--($(S1)+(0,1)-(xshift)$)node[bul]{};
\draw ($(S1)-(0,1)+(xshift)$)node[bul]{}--($(S1)+(0,0)+(xshift)$)to[out=90, in=-90](B2)node[bul]{};
\draw ($(S3)-(0,1)-(xshift)$)node[bul]{}--($(S3)+(0,0)-(xshift)$)to[out=90, in=-90](B3)node[bul]{};
\draw ($(S3)-(0,1)+(xshift)$)node[bul]{}to[out=90, in=-90](D)--($(D)+(0, 1)$)node[bul]{};
\draw ($(S4)-(0,1)-(xshift)$)node[bul]{}to[out=90, in=-90](D)--($(D)+(0, 1)$)node[bul]{};
\draw ($(S4)-(0,1)+(xshift)$)node[bul]{}--($(S4)+(0,1)+(xshift)$)node[bul]{};

\draw[yellow, thick] ($(S1)-(xshift)-2*(ep)$)--($(S4)+(xshift)+2*(ep)$);
\end{tikzpicture}
$$
The remaining coherence morphisms are identities.

The main part of the proof is to describe the values of $N_H$ on the morphisms of $\ms{BrCom}$. Let us do it and illustrate it  with the example

$$N_H\bigl(
\begin{tikzpicture}[scale=0.25,baseline=0.4cm]
\node(b1)at(-3,0)[bul]{};
\node(b2)at(-1,0)[bul]{};
\node(b3)at(1,0)[bul]{};
\node(b4)at(3,0)[bul]{};
\node(t1)at(-2,4)[bul]{};
\node(t2)at(0,4)[bul]{};
\node(t3)at(2,4)[bul]{};
\draw(b3) to[out=110,in=-70] (t1);
\draw[line width=1ex,white](b1)to[out=70,in=-110](t2) (b2)to[out=70,in=-90](t3);
\draw(b1)to[out=70,in=-110](t2) (b2)to[out=70,in=-90](t3) (b4)to[out=90,in=-90](t3);
\end{tikzpicture}
\bigr)\colon H^{\otimes3}\to H^{\otimes2}.$$

If $\phi:\bullet^m\to\bullet^n$, we get $N_H(\phi):H^{\otimes(m-1)}\to H^{\otimes(n-1)}$ as follows:
\begin{itemize}
		\item We start by drawing $m$ points on a horizontal line $\ell$ in the plane. There are $m-1$ intervals between these points, each of them carrying one copy of $H$  (all together representing $H^{\otimes(m-1)}$).
$$
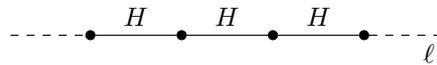
\begin{tikzpicture}[xscale=1.5,yscale=0.25]
\node(x1) at(-1.2,0)[bul]{};
\node(x2) at(-.4,0)[bul]{};
\node(x3) at(.4,0)[bul]{};
\node(x4) at(1.2,0)[bul]{};
\draw[dashed](-1.9,0)--(x1);
\draw[dashed](x4)--(1.9,0)node[below left]{$\ell$};
\draw(x1)--node[above]{$H$}(x2)--node[above]{$H$}(x3)--node[above]{$H$}(x4);
\end{tikzpicture}
$$
		\item The morphism (braid) $\phi$ can be interpreted as an isotopy of the plane between the identity and a diffeomorphism $\psi$ of $\R^2$, such that the isotopy stays the identity outside a large disk. The diffeomorphism $\psi$ brings the $m$ points to $n$ different landing pads, again situated along a horizontal line.
$$
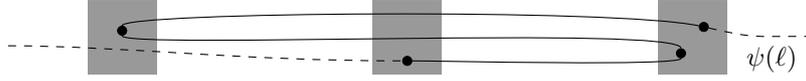
\begin{tikzpicture}[xscale=1.5,yscale=0.25]
 \fill[black!40] (-2.8,-1.5)--(-2.2,-1.5)--(-2.2,2.5)--(-2.8,2.5)--cycle;
\fill[black!40] (-0.3,-1.5)--(0.3,-1.5)--(0.3,2.5)--(-0.3,2.5)--cycle;
\fill[black!40] (2.8,-1.5)--(2.2,-1.5)--(2.2,2.5)--(2.8,2.5)--cycle;

\node(1) at (0,-0.8)[bul]{};
\node(2) at (2.4,-0.4)[bul]{};
\node(3) at (-2.5,0.8)[bul]{};
\node(4) at (2.6,1)[bul]{};
\draw[dashed](-3.5,0)..controls+(1,0)and+(-1.5,0)..(1);
\path[draw, name path = seg1] (1)..controls+(0.3,0)and+(0,-0.7)..(2);
\path[draw, name path = seg2] (2)..controls+(0,1.5)and+(0,-1)..(3);
\path[draw, name path = seg3] (3)..controls+(0,1)and+(-0.75,1)..(4);
\draw[dashed](4)..controls+(0.45,-0.6)and+(-0.5,0)..(3.5,0.5)node[below left]{$\psi(\ell)$};

\end{tikzpicture}
$$

		\item Now we draw $n-1$  vertical  lines between the landing pads. We suppose that they intersect $\psi(\ell)$ transversely. If the $\psi$-image of an interval meets the vertical lines $k$ times, we apply the iterated coproduct $H\to H^{\otimes k}$ to the copy of $H$ associated to that interval (if $k = 0$, we use the counit). At this point we have one  $H$ for each intersection point, i.e.\ we produced a morphism $H^{\otimes(m-1)}\to H^{\otimes M}$ where $M$ is the number of the intersections.
$$
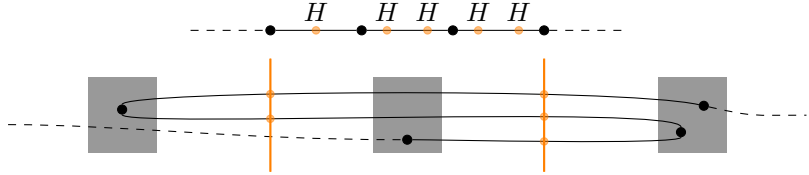
\begin{tikzpicture}[xscale=1.5,yscale=0.25]
\node(x1) at(-1.2,0)[bul]{};
\node(x2) at(-.4,0)[bul]{};
\node(x3) at(.4,0)[bul]{};
\node(x4) at(1.2,0)[bul]{};
\draw[dashed](-1.9,0)--(x1);
\draw[dashed](x4)--(1.9,0);
\draw(x1)--node[obul]{}node[above]{$H$}(x2) (x2)--node[obul, near start]{}node[above, near start]{$H$} node[obul, near end]{}node[above, near end]{$H$}(x3)(x3)--node[obul, near start]{}node[above, near start]{$H$} node[obul, near end]{}node[above, near end]{$H$}(x4);

\begin{scope}[yshift=-5cm]
\fill[black!40] (-2.8,-1.5)--(-2.2,-1.5)--(-2.2,2.5)--(-2.8,2.5)--cycle;
\fill[black!40] (-0.3,-1.5)--(0.3,-1.5)--(0.3,2.5)--(-0.3,2.5)--cycle;
\fill[black!40] (2.8,-1.5)--(2.2,-1.5)--(2.2,2.5)--(2.8,2.5)--cycle;

\node(1) at (0,-0.8)[bul]{};
\node(2) at (2.4,-0.4)[bul]{};
\node(3) at (-2.5,0.8)[bul]{};
\node(4) at (2.6,1)[bul]{};
\draw[dashed](-3.5,0)..controls+(1,0)and+(-1.5,0)..(1);
\path[draw, name path = seg1] (1)..controls+(0.3,0)and+(0,-0.7)..(2);
\path[draw, name path = seg2] (2)..controls+(0,1.5)and+(0,-1)..(3);
\path[draw, name path = seg3] (3)..controls+(0,1)and+(-0.75,1)..(4);
\draw[dashed](4)..controls+(0.45,-0.6)and+(-0.5,0)..(3.5,0.5);

\path[draw=orange, name path=vert1] (-1.2,-2.5)--(-1.2,3.5);
\path[draw=orange, name path=vert2] (1.2,-2.5)--(1.2,3.5);
\fill [orange, name intersections={of=seg1 and vert2}]
(intersection-1)  node[obul] {};
\fill [orange, name intersections={of=seg2 and vert2}]
(intersection-1) node[obul] {};
\fill [orange, name intersections={of=seg2 and vert1}]
(intersection-1) node[obul] {};
\fill [orange, name intersections={of=seg3 and vert1}]
(intersection-1) node[obul] {};
\fill [orange, name intersections={of=seg3 and vert2}]
(intersection-1) node[obul] {};
\path[draw=orange, thick] (-1.2,-2.5)--(-1.2,3.5);
\path[draw=orange, thick] (1.2,-2.5)--(1.2,3.5);
\end{scope}
\end{tikzpicture}
$$

		\item For each of the intersection points $P$ we find the total number of half-turns $k_P$ when moving along $\psi(\ell)$ (turns in the positive direction are counted positively, in the negative direction negatively). Then we apply $S^{k_P}$ to the $H$ corresponding to $P$.
$$
\begin{tikzpicture}[xscale=1.5,yscale=0.25]
\node(x1) at(-1.2,0)[bul]{};
\node(x2) at(-.4,0)[bul]{};
\node(x3) at(.4,0)[bul]{};
\node(x4) at(1.2,0)[bul]{};
\draw[dashed](-1.9,0)--node[below]{$k_P:$}(x1);
\draw[dashed](x4)--(1.9,0);
\draw(x1)--node[obul]{}node[below]{$\vphantom{k_P}0$}(x2) (x2)--node[obul, near start]{}node[below, near start]{$\vphantom{k_P}1$} node[obul, near end]{}node[below, near end]{$\vphantom{k_P}1$}(x3)(x3)--node[obul, near start]{}node[below, near start]{$\vphantom{k_P}0$} node[obul, near end]{}node[below, near end]{$\vphantom{k_P}0$}(x4);
\end{tikzpicture}
$$

		\item We move each of the $M$ $H$'s by the isotopy (using the braiding in $\C$), and finally multiply them along the vertical lines from the bottom to the top. Composing all these morphisms we get $N_H(\phi)\colon H^{\otimes(m-1)}\to H^{\otimes(n-1)}$. (The multiplication is done by first moving the intersection points to a horizontal position, with the bottom-most points on the left and the top-most on the right, and then multiplying the corresponding $H$'s.)
\end{itemize}

In our example we thus have (suppressing the braiding in $\C$ in the formula)

$$
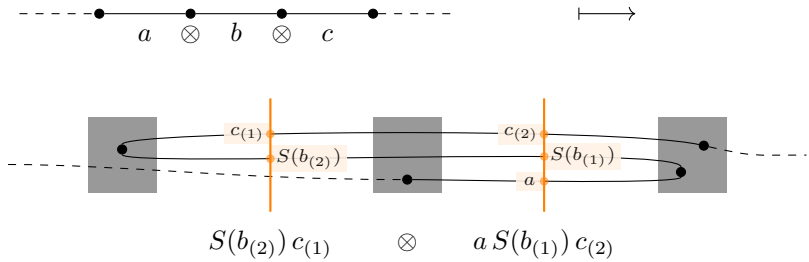
\begin{tikzpicture}[xscale=1.5,yscale=0.25]
\begin{scope}[xshift=-1.5cm, yshift=8cm]
\node(x1) at(-1.2,0)[bul]{};
\node(x2) at(-.4,0)[bul]{};
\node(x3) at(.4,0)[bul]{};
\node(x4) at(1.2,0)[bul]{};
\draw[dashed](-1.9,0)--(x1);
\draw[dashed](x4)--(1.9,0);
\draw(x1)--node[below]{$\vphantom{b}a$}(x2)node[below]{$\otimes\vphantom{b}$}--node[below]{$b$}(x3)node[below]{$\otimes\vphantom{b}$}--node[below]{$\vphantom{b}c$}(x4);
\draw[|->] (3,0)--(3.5,0);
\end{scope}
\fill[black!40] (-2.8,-1.5)--(-2.2,-1.5)--(-2.2,2.5)--(-2.8,2.5)--cycle;
\fill[black!40] (-0.3,-1.5)--(0.3,-1.5)--(0.3,2.5)--(-0.3,2.5)--cycle;
\fill[black!40] (2.8,-1.5)--(2.2,-1.5)--(2.2,2.5)--(2.8,2.5)--cycle;

\node(1) at (0,-0.8)[bul]{};
\node(2) at (2.4,-0.4)[bul]{};
\node(3) at (-2.5,0.8)[bul]{};
\node(4) at (2.6,1)[bul]{};
\draw[dashed](-3.5,0)..controls+(1,0)and+(-1.5,0)..(1);
\path[draw, name path = seg1] (1)..controls+(0.3,0)and+(0,-0.7)..(2);
\path[draw, name path = seg2] (2)..controls+(0,1.5)and+(0,-1)..(3);
\path[draw, name path = seg3] (3)..controls+(0,1)and+(-0.75,1)..(4);
\draw[dashed](4)..controls+(0.45,-0.6)and+(-0.5,0)..(3.5,0.5);

\path[draw=orange, name path=vert1] (-1.2,-2.5)--(-1.2,3.5);
\path[draw=orange, name path=vert2] (1.2,-2.5)--(1.2,3.5);
\fill [orange, name intersections={of=seg1 and vert2}]
(intersection-1)  node[left,black, fill=orange!10, fill opacity=0.9,inner xsep=3pt, inner ysep=2pt] {$\scriptstyle a$} node[obul] {};
\fill [orange, name intersections={of=seg2 and vert2}]
(intersection-1) node[right,black, fill=orange!10,fill opacity=0.9,inner xsep=2pt, inner ysep=1pt] {$\scriptstyle S(b_{(1)})$}node[obul] {};
\fill [orange, name intersections={of=seg2 and vert1}]
(intersection-1) node[right,black, fill=orange!10,fill opacity=0.9,inner xsep=2pt,inner ysep=1pt] {$\scriptstyle S(b_{(2)})$}node[obul] {};
\fill [orange, name intersections={of=seg3 and vert1}]
(intersection-1) node[left,black, fill=orange!10,fill opacity=0.9,inner xsep=2pt, inner ysep=1pt] {$\scriptstyle c_{(1)}$}node[obul] {};
\fill [orange, name intersections={of=seg3 and vert2}]
(intersection-1) node[left,black, fill=orange!10,fill opacity=0.9,inner xsep=2pt,inner ysep=1pt] {$\scriptstyle c_{(2)}$}node[obul] {};
\node at(-1.2,-4.5)[anchor=base] {$S(b_{(2)})\,c_{(1)}$};
\node at(1.2,-4.5)[anchor=base] {$a\,S(b_{(1)})\,c_{(2)}$};
\node at(0,-4.5)[anchor=base] {$\otimes\vphantom{S}$};
\path[draw=orange, thick] (-1.2,-2.5)--(-1.2,3.5);
\path[draw=orange, thick] (1.2,-2.5)--(1.2,3.5);

\end{tikzpicture}
$$
i.e.\  $N_H(\phi):H^{\otimes 3}\to H^{\otimes 2}$ is
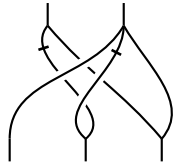
$$
\begin{tikzpicture}[thick]
\coordinate (1) at (-1,0);
\coordinate (2) at (0,0);
\coordinate (3) at (1,0);
\coordinate (a) at (-0.5,1.5);
\coordinate (b) at (0.5,1.5);
\draw (1)--+(0,-0.3) (2)--+(0,-0.3) (3)--+(0,-0.3);
\draw (a)--+(0,0.3) (b)--+(0,0.3);

\draw (2) to[out=40, in=-130] node[anti, pos=0.8]{} (a);
\draw (3) to[out=130, in=-50] (a);

\draw[line width=1ex,white] (1) to[out=90, in=-120] (b);
\draw[line width=1ex,white] (2) +(-0.01,0.01) to[out=140, in=-90] (b);

\draw (1) to[out=90, in=-120] (b);
\draw (2) to[out=140, in=-90] node[anti, pos=0.8]{} (b);
\draw (3) to[out=50, in=-60] (b);
\end{tikzpicture}
$$

To see that $N_H(\phi)$ is well defined, i.e.\ independent of the details of the isotopy and of the transversals, we need to verify that it is invariant under moves  of the type
$$
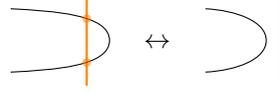
\begin{tikzpicture}[yscale=0.6,baseline=-0.1cm]
\path[draw=orange, thick, name path=vert] (0,-1)--(0,1);
\path[draw=black, name path=ell] (-1,-0.7) to[out=5,in=-90] (0.3,0) to[out=90,in=-5] (-1,0.7);
\fill [orange, name intersections={of=vert and ell}]
(intersection-1)  node[obul] {};
\fill [orange, name intersections={of=vert and ell}]
(intersection-2)  node[obul] {};
\end{tikzpicture}
\quad
\leftrightarrow
\quad
\begin{tikzpicture}[yscale=0.6,baseline=-0.1cm]
\path[draw=orange, thick, name path=vert] (0,-1)--(0,1);
\path[draw=black, name path=ell] (-1,-0.7) to[out=5,in=-90] (-0.2,0) to[out=90,in=-5] (-1,0.7);
\end{tikzpicture}
$$
This invariance follows from the defining property of the antipode.

 The fact that $N_H$ is a functor, i.e.\ that $N_H(\phi_1\circ\phi_2)=N_H(\phi_1)\circ N_H(\phi_2)$, follows from $H$ being a bialgebra. 		Namely, let $m^{(p)}\colon H^{\otimes p} \to H$ denote the iterated product
 $$m^{(0)} = \eta,\quad m^{(p+1)} = m \circ (m^{(p)}\otimes\id_H).$$
and  $\Delta^{(q)}\colon H\to H^{\otimes q}$   the iterated coproduct 
$$\Delta^{(0)} = \epsilon,\quad\Delta^{(q+1)} = (\Delta^{(q)}\otimes \id_H)\circ\Delta.$$
		Then for each $p, q$ we have
\begin{equation}\label{bialgbig}
(m^{(q)})^{\otimes p}\circ \tau_{p, q} \circ(\Delta^{(p)})^{\otimes q}  =   \Delta^{(p)} \circ m^{(q)} 
\end{equation}
where $\tau_{p, q}\colon H^{\otimes pq} \to H^{\otimes pq}$ is given by the braid which reshuffles $q$ groups, each made of
		$p$ strands, into $p$ groups, each made of $q$ strands, by taking the first element of each group together, then second etc., using overcrossings only.
$$
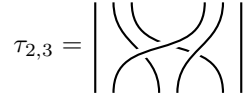
\tau_{2,3}=\,
\begin{tikzpicture}[thick, baseline=0.55cm, yscale=1.2, xscale=0.8]
\coordinate (a1) at (-1.2,0);
\coordinate (a2) at (-0.9,0);
\coordinate (b1) at (-0.15,0);
\coordinate (b2) at (0.15,0);
\coordinate (c1) at (0.9,0);
\coordinate (c2) at (1.2,0);
\coordinate (A1) at (-1.2,1);
\coordinate (A2) at (-0.9,1);
\coordinate (A3) at (-0.6,1);
\coordinate (B3) at (1.2,1);
\coordinate (B2) at (0.9,1);
\coordinate (B1) at (0.6,1);
\draw (a1) to[out=90,in=-90] (A1) (b1) to[out=90,in=-90] (A2) (c1) to[out=90,in=-90] (A3);
\draw[white, line width=1ex] (a2) to[out=90,in=-90] (B1) (b2) to[out=90,in=-90] (B2) (c2) to[out=90,in=-90] (B3);
\draw (a2) to[out=90,in=-90] (B1) (b2) to[out=90,in=-90] (B2) (c2) to[out=90,in=-90] (B3);
\end{tikzpicture}
$$
The expression for $N_H(\phi_1\circ\phi_2)$ differs from the expression for $N_H(\phi_1)\circ N_H(\phi_2)$ only by replacing  RHSs of \eqref{bialgbig} with LHSs, and so they are equal. (If we use the ``$\Delta$ is horizontal and $m$ vertical'' convention as in the construction of $N_H$ then the identity \eqref{bialgbig} is the commutativity of the diagram
$$
\begin{tikzpicture}
\foreach \j in {0,0.2,0.4} \node[obul] at (0.1,\j) {}; 
\draw[->] (0.4,0.2)--node[above]{$(\Delta^{(p)})^{\otimes q}$}(2.2,0.2);
\draw[->] (0.1,-0.2)--node[left]{$m^{(q)}$}(0.1,-1.1);

\begin{scope}[xshift=2.5cm]
\foreach \i in {0,0.2}
  \foreach \j in {0,0.2,0.4}
    \node[obul] at (\i,\j) {};
\draw[->] (0.1,-0.2)--node[right]{$(m^{(q)})^{\otimes p}$}(0.1,-1.1);
\end{scope}

\begin{scope}[yshift=-1.5cm]
\node[obul] at (0.1,0.2) {};
\draw[->] (0.4,0.2)--node[above]{$\Delta^{(p)}$}(2.2,0.2);
\end{scope}

\begin{scope}[xshift=2.5cm, yshift=-1.5cm]
\foreach \i in {0,0.2} \node[obul] at (\i,0.2) {};
\end{scope}

\node at (6.5,-0.6) {($p=2,q=3$)};

\node at (-2,0){};

\end{tikzpicture}
$$
which is much more enlightening in our context.)

 Finally, the fact that $N_H$ is braided lax monoidal and that it satisfies the strict nerve condition is evident. Moreover the construction $H\mapsto N_H$ is clearly functorial.

If $N\colon\ms{BrCom}\to\C$ satisfies the strict nerve condition then checking that $H_N\vcentcolon=N({\bullet\bullet})$ (with the operations given in the theorem) is a Hopf algebra is a simple manipulation with diagrams; it is also a special case of Theorem 1 from \cite{ja}. Again the construction $N\mapsto H_N$ is functorial.

 One easily checks that $\cramped{H_{N_H}}=H$ (as Hopf algebras). To finish the proof, we need to verify that $\cramped{N_{H_N}}=N$, and it is sufficient to do it for the morphisms of the type
$$
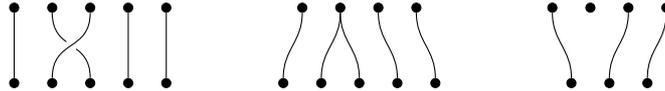
\begin{tikzpicture}[xscale=0.5]
\node[bul](1) at (0,0) {};
\node[bul](2) at (1,0) {};
\node[bul](3) at (2,0) {};
\node[bul](4) at (3,0) {};
\node[bul](5) at (4,0) {};
\node[bul](1a) at (0,1) {};
\node[bul](2a) at (1,1) {};
\node[bul](3a) at (2,1) {};
\node[bul](4a) at (3,1) {};
\node[bul](5a) at (4,1) {};
\draw(3) to[out=90, in=-90] (2a);
\draw[line width=1ex,white] (2) to[out=90,in=-90] (3a);
\draw(2) to[out=90,in=-90] (3a);
\draw(1)--(1a) (4)--(4a) (5)--(5a);
\end{tikzpicture}
\qquad\qquad
\begin{tikzpicture}[xscale=0.5]
\node[bul](1) at (0,0) {};
\node[bul](2) at (1,0) {};
\node[bul](3) at (2,0) {};
\node[bul](4) at (3,0) {};
\node[bul](5) at (4,0) {};
\node[bul](1a) at (0.5,1) {};
\node[bul](2a) at (1.5,1) {};
\node[bul](3a) at (2.5,1) {};
\node[bul](4a) at (3.5,1) {};
\draw(1) to[out=90, in=-90] (1a);
\draw(2) to[out=90, in=-90] (2a);
\draw(3) to[out=90, in=-90] (2a);
\draw(4) to[out=90, in=-90] (3a);
\draw(5) to[out=90, in=-90] (4a);
\end{tikzpicture}
\qquad\qquad
\begin{tikzpicture}[xscale=0.5]
\node[bul](1) at (0,0) {};
\node[bul](2) at (1,0) {};
\node[bul](3) at (2,0) {};
\node[bul](1a) at (-0.5,1) {};
\node[bul](2a) at (0.5,1) {};
\node[bul](3a) at (1.5,1) {};
\node[bul](4a) at (2.5,1) {};
\draw(1) to[out=90, in=-90] (1a);
\draw(2) to[out=90, in=-90] (3a);
\draw(3) to[out=90, in=-90] (4a);
\end{tikzpicture}
$$
as they generate $\ms{BrCom}$. Checking it is again a straightforward calculation.

\begin{rem}\label{rem:YD}
The functor $N_H$ can alternatively be constructed as follows. If $\mc D$ is the category of $H$-dimodules in $\C$ (the Yetter-Drinfeld category) then $H$ is naturally a commutative algebra in $\mc D$, and thus we have a braided strong monoidal functor $F\colon\ms{BrCom}\to\mc D$, $F(\bullet^n)=H^{\otimes n}$. $H$-coinvariants (of the $H$-coaction on dimodules) is then a braided lax monoidal functor $\mc D\to\mc C$ (in general this functor might not be defined on whole of $\mc D$, but it is defined on the image of $F$), and we can define $N_H$ as the composition of these two braided lax monoidal functors.

In more detail, $H$-dimodules are objects $X$ of $\C$ equipped with a (left) $H$-action and $H$-coaction, such that for any $H$-module $Y$ the $\C$-morphism 
$$
\tilde \beta_{X,Y}:X\otimes Y\to Y\otimes X,\qquad
\tilde \beta_{X,Y}=
\begin{tikzpicture}[baseline=0.4cm, xscale=0.8]
\draw[thick,gray](1,0)node[below,black]{$Y$} to[out=105,in=-75] coordinate[pos=0.8](a) (0,1)node[above,black]{$Y$};
\draw[line width=1ex,white](0,0)--(1,1);
\draw[thick,gray](0,0)node[below,black]{$X$} to[out=75,in=-105] coordinate[pos=0.2](b) (1,1)node[above,black]{$X$};
\draw[thick](a)--(b);
\end{tikzpicture}
$$
is a morphism of $H$-modules. The category $\mc D$ of $H$-dimodules is braided via $\tilde\beta$. $H$ is an $H$-comodule via $\Delta$ and there is a unique (adjoint) $H$-module structure on $H$ such that the action $H\otimes Y\to Y$ is an $H$-module morphism for every $H$-module $Y$. This makes $H$ to an object of $\mc D$, and its algebra structure makes it to a commutative algebra in $\mc D$. We have
$$\tilde \beta_{H,H}=
\begin{tikzpicture}[baseline=0.4cm]
\draw[thick](0.7,0)to[out=90,in=-90] coordinate[pos=0.9](x) (0,1);
\draw[line width=1ex,white](0,0)--(0.7,1);
\draw[thick](0,0)to[out=90,in=-90] coordinate[pos=0.3](a)coordinate[pos=0.65](b) (0.7,1);
\draw[thick](a)to[out=120,in=-120](x) (b)--node[anti]{}(x);
\end{tikzpicture}
$$
and for any morphism  $\phi\colon\bullet^p\to\bullet^q$  in $\ms{BrCom}$ we get
$N_H(\phi)=\kappa_q\circ F(\phi)\circ \iota_p$
where the morphisms $\iota_p$ and $\kappa_q$ are defined below (see Equation \eqref{PNHmor}).
\end{rem}

\subsection{Nerves of Poisson Hopf algebras (proof of Theorem \ref{thm:inerve})}\label{sec:pfi}
As in \S\ref{sec:pfbr}, we shall suppose that $\C$ is strict monoidal and prove the theorem in the form saying that there is an isomorphism of categories between the category of Poisson Hopf algebras in $\C$ and the category of i-braided lax monoidal functors $N\colon\ms{iCom}(R)\to\C$ satisfying the strict nerve condition.
 
The main part of the proof is again a construction of a suitable $N_H\colon\ms{iCom}(R)\to\C$ out of a Poisson Hopf algebra $H\in\C$. We already know the restriction of $N_H$ to the subcategory $\ms{Com}\subset\ms{iCom}(R)$, which is the nerve of the commutative Hopf algebra $H$, and also its coherence morphisms \eqref{nervecm}.

We shall build the functor $N_H$ in two steps. 
Let $r\in\on{End}_\C(H\otimes H)$ be 
$$r=-\,
\begin{tikzpicture}[baseline=0.4cm]
\draw[thick] (0,0)--node(p)[bbul,pos=0.5]{} coordinate[pos=0.8](m) (0,1);
\draw[thick] (0.5,0)--(0.5,1) (0.5,0.2)--(p) (0.5,0.5)--node[anti]{} (m);
\end{tikzpicture}
$$
and let
$$\tau=r+r^\textit{op}+t_{H,H}$$
where $r^\textit{op}=\sigma_{H,H}\circ r\circ\sigma_{H,H}$.

\begin{lem}\label{lem:PH1}
There is a unique $R$-linear symmetric strict monoidal functor $$F_H\colon\ms{iCom}(R)\to\C$$ satisfying $F_H(\bullet)=H$, $F_H(\uln m)=m$, $F_H(\uln\eta)=\eta$, $F_H(t_{{\bullet,\bullet}})=\tau$, where $m$ and $\eta$ are the product and the unit of $H$.
\end{lem}
On objects the functor $F_H$ is given by $F_H(\bullet^n)=H^{\otimes n}$ and on a typical morphism by 
$$
\begin{tikzpicture}[baseline=0.5cm]
\node(A1) at (0,0)[bul]{};
\node(A2) at (.5,0)[bul]{};
\node(A3) at (1,0)[bul]{};
\node(A4) at (1.5,0)[bul]{};
\node(B1) at (.25,1.3)[bul]{};
\node(B2) at (.75,1.3)[bul]{};
\node(B3) at (1.25,1.3)[bul]{};
\draw (A1)..controls +(0,.5) and +(0,-.5)..node(a1)[very near start, rbul]{}(B2);
\draw (A2)..controls +(0,.5) and +(0,-.5)..node(a2)[very near start, rbul]{}node(c1)[near end, rbul]{}(B1);
\draw (A3)..controls +(0,.5) and +(0,-.5)..node(b1)[near start, rbul]{}(B2);
\draw (A4)..controls +(0,.5) and +(0,-.5)..node(b2)[near start, rbul]{}node(c2)[near end, rbul]{}(B3);
\draw[red](a1)--(a2) (b1)--(b2) (c1)--(c2);
\node(X)at(0,-0.2){};
\end{tikzpicture}
\quad\xmapsto{F_H}\quad
\begin{tikzpicture}[baseline=0.5cm]
\coordinate(A1) at (0,0.05);
\coordinate(A2) at (.5,0.05);
\coordinate(A3) at (1,0.05);
\coordinate(A4) at (1.5,0.05);
\coordinate(B1) at (.25,1.25);
\coordinate(B2) at (.75,1.25);
\coordinate(B3) at (1.25,1.25);
\draw[thick] (A1)..controls +(0,.5) and +(0,-.5)..node(a1)[very near start, rbul]{}(B2);
\draw[thick] (A2)..controls +(0,.5) and +(0,-.5)..node(a2)[very near start, rbul]{}node(c1)[near end, rbul]{}(B1);
\draw[thick] (A3)..controls +(0,.5) and +(0,-.5)..node(b1)[near start, rbul]{}(B2);
\draw[thick] (A4)..controls +(0,.5) and +(0,-.5)..node(b2)[near start, rbul]{}node(c2)[near end, rbul]{}(B3);
\draw[red, decorate,decoration={zigzag,amplitude=1pt, segment length=2pt}](a1)--(a2) (b1)--(b2) (c1)--(c2);
\draw[thick](A1)--+(0,-0.1) (A2)--+(0,-0.1) (A3)--+(0,-0.1) (A4)--+(0,-0.1)  (B1)--+(0,0.1) (B2)--+(0,0.1) (B3)--+(0,0.1);
\end{tikzpicture}
$$
using the notation $\tau=
\begin{tikzpicture}[scale=0.3, baseline=1.5pt]
\draw[thick](0,0)--node(a)[rbul]{} (0,1) (2,0)--node(b)[rbul]{} (2,1);
\draw[red, decorate,decoration={zigzag,amplitude=1pt, segment length=2pt}] (a)--(b);
\end{tikzpicture}$.

Out of $F_H$ we  now build $N_H$ as follows.  On objects it is defined by
\begin{equation}\label{PNHob}
N_H(\bullet^n)=H^{\otimes(n-1)},\quad N_H(1)=1_\C.
\end{equation}
 To define $N_H$ on morphisms, let
 $\iota_n\colon N_H(\bullet^n)\to F_H(\bullet^n)$ be the morphisms
$$\iota_0=\id\colon1_\C\to1_\C,\quad \iota_1=\eta\colon1_\C\to H$$
and for $n>1$
$$\iota_n= 
\begin{tikzpicture}[scale=0.6, baseline=-0.4cm]
\foreach \x  in {0,...,3}
   \coordinate(\x) at (\x,0);
\foreach \x  in {0,1,2}
   \coordinate(a\x) at (\x.5,-1);
\foreach \x  in {0,1,2}
   \draw[thick] (\x)--node[anti]{}(a\x) (a\x)--+(0.5,1) (a\x)--+(0,-0.5);
\foreach \x  in {1,2}
   \draw[thick] (\x)--+(0,0.5);
\draw[thick] (0) to[out=110,in=-90] (-0.05,0.5) (3) to[out=70,in=-90] (3.05,0.5);
\end{tikzpicture}
\quad(n=4).
$$
Then $N_H$ is determined by the requirement that $\iota_n$'s form a natural transformation $N_H\to F_H$. If
$\kappa_n\colon F_H(\bullet^{ n})\to N_H(\bullet^{ n})$ satisfy $\kappa_n\circ \iota_n=\id$ then for any morphism  $\phi\colon\bullet^p\to\bullet^q$ in $\ms{iCom}(R)$ we have
\begin{equation}\label{PNHmor}
N_H(\phi)=\kappa_q\circ F_H(\phi)\circ \iota_p.
\end{equation}
For the morphisms $\kappa_n$ we can take $\kappa_0=\id$ and for $n>0$
$$ 
\kappa_n=\;
\begin{tikzpicture}[baseline=0.3cm]
\draw[thick](0,0)--(0,0.7)node[tbul]{};
\draw[thick](0.3,0)--(0.3,1.05) (0.6,0)--(0.6,1.05) (0.9,0)--(0.9,1.05) (1.2,0)--(1.2,1.05);
\draw[thick](0.3,0.9)--(0.6,0.75) (0.6,0.6)--(0.9,0.45) (0.9,0.3)--(1.2,0.15);
\end{tikzpicture}
\qquad(n=5)
$$
where we use the notation $\epsilon=%
\begin{tikzpicture}[scale=0.2, baseline]
\draw[thick] (0,0)--(0,1)node[tbul]{};
\end{tikzpicture}$.

As promised, we have the following result. 
 
\begin{lem}\label{lem:PH2}
$N_H$, defined on objects by \eqref{PNHob} and on morphisms by \eqref{PNHmor}, is an i-braided lax monoidal functor satisfying the strict nerve condition.
\end{lem}

Let us now prove the two lemmas.

\begin{proof}[Proof of Lemma \ref{lem:PH1}]
We need to verify the identities
$$
\begin{tikzpicture}[baseline=0.45cm]
\draw[thick](-0.3,0) to[out=90, in=-120] node[rbul](a){} (0,.8) (0.3,0) to[out=90, in=-60] node[rbul](b){} (0,.8)--(0,1.2);
\draw[red, decorate,decoration={zigzag,amplitude=1pt, segment length=2pt}] (a)--(b);
\end{tikzpicture}=0,
\qquad
\begin{tikzpicture}[baseline=0.45cm]
\draw[thick](-0.2,0) to[out=90, in=-120] (0,.8) (0.2,0) to[out=90, in=-60]  (0,.8)-- node[rbul](a){}(0,1.2) (1,0)--(1,.8)--node[rbul](b){}(1,1.2);
\draw[red, decorate,decoration={zigzag,amplitude=1pt, segment length=2pt}] (a)--(b);
\end{tikzpicture}
\ =\ 
\begin{tikzpicture}[baseline=0.45cm]
\draw[thick](-0.2,0) to[out=90, in=-120] node[rbul](a){} (0,.8) (0.2,0) to[out=90, in=-60]  (0,.8)--(0,1.2) (1,0)--node[rbul](b){}(1,.83)--(1,1.2);
\draw[red, decorate,decoration={zigzag,amplitude=1pt, segment length=2pt}] (a)--(b);
\end{tikzpicture}
\ +\ 
\begin{tikzpicture}[baseline=0.45cm]
\draw[thick](-0.2,0) to[out=90, in=-120] (0,.8) (0.2,0) to[out=90, in=-60] node[rbul](a){}  (0,.8)--(0,1.2) (1,0)--node[rbul](b){}(1,.83)--(1,1.2);
\draw[red, decorate,decoration={zigzag,amplitude=1pt, segment length=2pt}] (a)--(b);
\end{tikzpicture}
$$ 
and that $\tau_{ij}$'s satisfy the Drinfeld-Kohno relations, where $\tau_{ij}\in\on{End}(H^{\otimes n})$ is $\tau$ acting on the $i$'th and $j$'th $H$ in $H^{\otimes n}$.

The first identity follows from \eqref{skewmod}. The second identity, saying that $\tau$ is a biderivation, follows from the fact that both $t_{H,H}$ and $r$ are biderivations. 

To show the Drinfeld-Kohno relations, let us first prove that $\tau$ is a morphism of $H$-comodules, i.e.\ that
\begin{equation}\label{tau-comod}
\begin{tikzpicture}[baseline=0.3cm]
\draw[thick](0,0)--node[rbul,pos=0.2](a){}(0,1);
\draw[thick](0.5,0)--node[rbul,pos=0.2](b){}(0.5,1);
\draw[red, decorate,decoration={zigzag,amplitude=1pt, segment length=2pt}] (a)--(b);
\draw[thick](0,0.35) to[out=120,in=-140](-0.3,0.85)--(-0.3,1);
\draw[thick](0.5,0.35) to[out=140,in=-40](-0.3,0.85);
\end{tikzpicture}
\ =\ %
\begin{tikzpicture}[baseline=0.3cm]
\draw[thick](0,0)--node[rbul,pos=0.8](a){}(0,1);
\draw[thick](0.5,0)--node[rbul,pos=0.8](b){}(0.5,1);
\draw[red, decorate,decoration={zigzag,amplitude=1pt, segment length=2pt}] (a)--(b);
\draw[thick](0,0.35) to[out=120,in=-140](-0.3,0.85)--(-0.3,1);
\draw[thick](0.5,0.35) to[out=140,in=-40](-0.3,0.85);
\end{tikzpicture}
\end{equation}
This follows from
$$
\begin{tikzpicture}[baseline=0.7cm]
\draw[thick] (0,0)--(0,1.5) (0.5,0)--(0.5,1.5);
\draw[thick] (0.5,0.2)--(0,0.5) node[bbul]{};
\draw[thick] (0.5,0.4)--node[anti]{}(0,0.7);
\draw[thick](0,0.9) to[out=120,in=-140](-0.3,1.35)--(-0.3,1.5);
\draw[thick](0.5,0.9) to[out=140,in=-40](-0.3,1.35);
\end{tikzpicture}
\ =
\begin{tikzpicture}[baseline=0.7cm]
\draw[thick] (0,0)--(0,1.5) (0.5,0)--(0.5,1.5);
\draw[thick](0,0.4) to[out=120,in=-140](-0.3,1.2)node[bbul](a){}--(-0.3,1.5);
\draw[thick](0.5,0.4) to[out=140,in=-40](a);
\end{tikzpicture}
\ +
\begin{tikzpicture}[baseline=0.7cm]
\draw[thick] (0,0)--(0,1.5) (0.5,0)--(0.5,1.5);
\draw[thick] (0.5,0.8)--(0,1.1) node[bbul]{};
\draw[thick] (0.5,1)--node[anti]{}(0,1.3);
\draw[thick](0,0.4) to[out=120,in=-140](-0.3,0.85)--(-0.3,1.5);
\draw[thick](0.5,0.4) to[out=140,in=-40](-0.3,0.85);
\end{tikzpicture}
\ +
\begin{tikzpicture}[baseline=0.7cm]
\draw[thick] (0,0)--(0,1.5) (0.5,0)--(0.5,1.5);
\draw[thick](0,0.4) to[out=120,in=-140](-0.3,1.2)--(-0.3,1.5);
\draw[thick](0.5,0.4)--(-0.3,1.2);
\draw[red](0,0.65)node[rbul]{}--(0.25,0.65)node[rbul]{};
\end{tikzpicture}
$$
which, in turn, is the identity \eqref{pHA} plugged into the gray rectangle of
$$
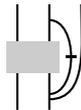
\begin{tikzpicture}[scale=0.7]
\draw[thick] (0.2,-0.5)--(0.2,1.5) (0.8,-0.5)--(0.8,1.5);
\fill[black!20!white] (0,0.2)rectangle(1,0.8);
\draw[thick] (0.8,-0.2) to[out=0,in=0] node[anti]{}(0.8,1.2);
\draw[thick] (0.8,-0.3) to[out=0,in=-90](1.4,1.5);
\end{tikzpicture}
$$

Let us introduce the notation
$$
\begin{tikzpicture}[baseline=0.4cm]
\draw[thick] (0,0)--node(a)[tv,pos=0.6]{} (0,1);
\draw[thick] (0.5,0) to[out=90, in=0](a);
\end{tikzpicture}
=
\begin{tikzpicture}[baseline=0.4cm]
\draw[thick] (0,0)--node(p)[bbul,pos=0.5]{} coordinate[pos=0.8](m) (0,1);
\draw[thick] (0.5,0)--(0.5,0.2)--(p) (0.5,0.2) to[out=60,in=-30]node[anti, pos=0.6]{} (m);
\end{tikzpicture}
\qquad\text{so that}\quad
r=-\ %
\begin{tikzpicture}[baseline=0.4cm]
\draw[thick] (0,0)--node(a)[tv,pos=0.6]{} (0,1);
\draw[thick] (0.5,0)--(0.5,1);
\draw[thick] (0.5,0.3) to[out=120, in=0](a);
\end{tikzpicture}
$$
The Jacobi identity \eqref{Jacobi} implies after a straightforward calculation
$$
\begin{tikzpicture}[baseline=0.3cm]
\draw[thick] (0,0)--node(a)[tv, pos=0.35]{} node(b)[tv, pos=0.65]{} (0,1);
\draw[thick] (0.3,0)to[out=90, in=0](a) (0.6,0)to[out=90, in=0](b);
\end{tikzpicture}
\ -\ %
\begin{tikzpicture}[baseline=0.3cm]
\draw[thick] (0,0)--node(a)[tv, pos=0.35]{} node(b)[tv, pos=0.65]{} (0,1);
\draw[thick] (0.3,0)to[out=90, in=0](b) (0.6,0)to[out=90, in=0](a);
\end{tikzpicture}
\ =\ %
\begin{tikzpicture}[baseline=0.3cm]
\draw[thick] (0,0)--node(b)[tv, pos=0.65]{} (0,1);
\node[bbul] (p) at (0.45,0.3) {};
\draw[thick] (0.3,0)to[out=90, in=-135](p) (0.6,0)to[out=90, in=-45](p)%
  (p)to[out=90, in=0] (b);
\end{tikzpicture}
$$
which, in turn, implies
\begin{equation}\label{tau-mod}
\begin{tikzpicture}[baseline=0.4cm]
\draw[thick](0,0)--node[rbul,pos=0.4](a){} node[tv,pos=0.7](tv){}(0,1) (0.3,0)--node[rbul,pos=0.4](b){}(0.3,1);
\draw[red, decorate,decoration={zigzag,amplitude=1pt, segment length=2pt}](a)--(b);
\draw[thick](0.6,0)to[out=90,in=0](tv);
\end{tikzpicture}
\ +\ %
\begin{tikzpicture}[baseline=0.4cm]
\draw[thick](0,0)--node[rbul,pos=0.4](a){}(0,1) (0.3,0)--node[rbul,pos=0.4](b){} node[tv,pos=0.7](tv){}(0.3,1);
\draw[red, decorate,decoration={zigzag,amplitude=1pt, segment length=2pt}](a)--(b);
\draw[thick](0.6,0)to[out=90,in=0](tv);
\end{tikzpicture}
\ =\ %
\begin{tikzpicture}[baseline=0.4cm]
\draw[thick](0,0)--node[rbul,pos=0.7](a){} node[tv,pos=0.4](tv){}(0,1) (0.3,0)--node[rbul,pos=0.7](b){}(0.3,1);
\draw[red, decorate,decoration={zigzag,amplitude=1pt, segment length=2pt}](a)--(b);
\draw[thick](0.6,0)to[out=90,in=0](tv);
\end{tikzpicture}
\ +\ %
\begin{tikzpicture}[baseline=0.4cm]
\draw[thick](0,0)--node[rbul,pos=0.7](a){}(0,1) (0.3,0)--node[rbul,pos=0.7](b){} node[tv,pos=0.4](tv){}(0.3,1);
\draw[red, decorate,decoration={zigzag,amplitude=1pt, segment length=2pt}](a)--(b);
\draw[thick](0.6,0)to[out=90,in=0](tv);
\end{tikzpicture}
\end{equation}
 
We can finally prove the non-trivial Drinfeld-Kohno relation $[\tau_{12}+\tau_{13},\tau_{23}]=0$. It is the sum of the identities
\begin{align*}
[r_{12}+r_{13},\tau_{23}]&=0\\
[r^\textit{op}_{\;12}+r^\textit{op}_{\;13},\tau_{23}]&=0\\
[t_{12}+t_{13},\tau_{23}]&=0.
\end{align*}
The first one is 
$$
\begin{tikzpicture}[baseline=0.7cm]
\draw[thick](0,0)--node[tv,pos=0.8](tv){}(0,1.5);
\draw[thick](0.5,0)--node[rbul,pos=0.8](a){}(0.5,1.5) (1,0)--node[rbul,pos=0.8](b){}(1,1.5);
\coordinate(m) at (0.2,1);
\draw[thick] (0.5,0.4)to[out=160,in=-120](m) (1,0.4)--(m) (m)to[out=90,in=0](tv);
\draw[red, decorate,decoration={zigzag,amplitude=1pt, segment length=2pt}](a)--(b);
\end{tikzpicture}
\ =\ %
\begin{tikzpicture}[baseline=0.7cm]
\draw[thick](0,0)--node[tv,pos=0.8](tv){}(0,1.5);
\draw[thick](0.5,0)--node[rbul,pos=0.15](a){}(0.5,1.5) (1,0)--node[rbul,pos=0.15](b){}(1,1.5);
\coordinate(m) at (0.2,1);
\draw[thick] (0.5,0.4)to[out=160,in=-120](m) (1,0.4)--(m) (m)to[out=90,in=0](tv);
\draw[red, decorate,decoration={zigzag,amplitude=1pt, segment length=2pt}](a)--(b);
\end{tikzpicture}
$$
 and so it follows from \eqref{tau-comod}. The second one follows from \eqref{tau-mod}, and the third one from $t_{12}+t_{13}=t_{H,H\otimes H}$ and from the fact that $t$ is a natural transformation.
\end{proof}

\begin{proof}[Proof of Lemma \ref{lem:PH2}]
Let us make $H$ to a left $H$-comodule via the coaction $\Delta\colon H\to H\otimes H$. Then $F_N(\bullet^n)=H^{\otimes n}$ becomes an $H$-comodule as well, and \eqref{tau-comod} implies that for any morphism $\phi\colon\bullet^p\to\bullet^q$ in $\ms{iCom}(R)$ the morphism $F_H(\phi)$ is a morphism of $H$-comodules.

The morphism $\iota_n$ is the equalizer of 
$$
\begin{tikzcd}[sep=large]
F_H(\bullet^n) \arrow[r, shift left, "\mathrm{coaction}" above] \arrow[r, shift right, "\eta\otimes\id" below] & H\otimes F_H(\bullet^n)
\end{tikzcd}
$$
(i.e.\ it gives us the coinvariants of the coaction). This shows that $N_H$ is indeed a functor.

The fact that $N_H$ is symmetric lax monoidal with the coherence morphisms $c_{\bullet^m,\bullet^n}$ given by  \eqref{nervecm} follows from the identity $\iota_{m+n}\circ c_{\bullet^m,\bullet^n}=\iota_m\otimes\iota_n$, which then means that $\iota_n$'s form a monoidal natural transformation $N_H\to F_H$.
The fact that $N_H$ satisfies the strict nerve condition is evident. Finally, the fact that $N_H$ is i-braided follows easily from seeing $\iota_n$ as the above-mentioned equalizer.
\end{proof}

Once $N_H$ is constructed, the rest of the proof is as in \S\ref{sec:pfbr}: The construction $H\mapsto N_H$ is functorial in $H$.
If $N\colon\ms{iCom}(R)\to\C$ satisfies the strict nerve condition then checking that $H_N\vcentcolon=N({\bullet\bullet})$ (with the operations given in the theorem) is a Poisson Hopf algebra, i.e.\ that \eqref{pHA} is satisfied, is a simple manipulation with diagrams. Again the construction $N\mapsto H_N$ is functorial. 

Let us verify that $\cramped{H_{N_H}}=H$ as Poisson Hopf algebras, i.e.\ that from $N_H$ we get back the Poisson bracket $p$ on $H$. By the definition of $N_H$ and its coherence morphisms we have
$$
\begin{tikzpicture}[xscale=0.6,yscale=0.7,baseline=0.5cm]
\coordinate (diff) at (0.8,0);
\coordinate (dy) at (0,0.3);
\coordinate(A1) at (0,0);
\coordinate(B1) at ($(A1)+(diff)$);
\coordinate(A2) at (2,0);
\coordinate(B2) at ($(A2)+(diff)$);
\coordinate(A3) at ($(B1)+(0,2)$);
\coordinate(B3) at ($(A2)+(0,2)$);
\coordinate(ep) at (0.3,0);
\coordinate(A) at ($(A3)-(dy)$);
\coordinate(B) at ($(B3)-(dy)$);

\fill[black!10] ($(A1)-(ep)$)..controls +(0,1) and +(0,-1)..($(A)-(ep)$)--($(A3)-(ep)$)%
--($(B3)+(ep)$)--($(B)+(ep)$)..controls +(0,-1) and +(0,1)..($(B2)+(ep)$)%
--($(A2)-(ep)$)..controls +(0,.5) and +(0,.5)..($(B1)+(ep)$);

\draw[black!30,thick] ($(A1)-(ep)$)..controls +(0,1) and +(0,-1)..($(A)-(ep)$)--($(A3)-(ep)$)%
($(B3)+(ep)$)--($(B)+(ep)$)..controls +(0,-1) and +(0,1)..($(B2)+(ep)$)%
($(A2)-(ep)$)..controls +(0,.5) and +(0,.5)..($(B1)+(ep)$);

\node(A1) at (A1)[bul]{};
\node(B1) at (B1) [bul]{};
\node(A2) at (A2) [bul]{};
\node(B2) at (B2) [bul]{};
\node(A3) at (A3) [bul]{};
\node(B3) at (B3) [bul]{};

\draw (A2)..controls +(0,1) and +(0,-1)..node(c1)[near start, rbul]{}(A);
\draw (B1)..controls +(0,1) and +(0,-1)..node(c2)[near start, rbul]{}(B);
\draw (A1)..controls +(0,1) and +(0,-1)..(A);
\draw (B2)..controls +(0,1) and +(0,-1)..(B);
\draw (B)--(B3);
\draw(A)--(A3);
\draw[red](c1)--(c2);
\end{tikzpicture}
=
\begin{tikzpicture}[baseline=0.32cm, yscale=0.7, xscale=0.8]
\coordinate(A1) at (0,0);
\coordinate(A2) at (1,0);
\coordinate(B1) at (-0.25,0.6);
\coordinate(B2) at (0.25,0.6);
\coordinate(B3) at (0.75,0.6);
\coordinate(B4) at (1.25,0.6);
\coordinate(C1) at (0,1.5);
\coordinate(C2) at (1,1.5);
\draw[thick](A1)--+(0,-0.25) (A1)to[out=120,in=-120]node[anti,pos=0.2]{}(C1) (A1)to[out=60,in=-140]node[rbul,pos=0.35](a){}(C2);
\draw[thick](A2)--+(0,-0.25) (A2)to[out=120,in=-40]node[anti,pos=0.2]{}node[rbul,pos=0.35](b){}(C1) (A2)to[out=60,in=-60](C2);
\draw[thick](C1)--+(0,0.1)node[tbul]{} (C2)--+(0,0.25);
\draw[red, decorate,decoration={zigzag,amplitude=1pt, segment length=2pt}](a)--(b);
\end{tikzpicture}
=-
\begin{tikzpicture}[baseline=0.5cm,yscale=0.7,xscale=0.8]
\draw[thick](0,0)to[out=90,in=-140]node[tv,sloped,pos=0.4](t){}(1,1.75);
\draw[thick](1,0)--(1,0.25)to[out=140,in=-30]node[anti]{}(t.south);
\draw[thick](1,0.25)to[out=60,in=-60](1,1.75)--(1,2);
\end{tikzpicture}
=
\begin{tikzpicture}[baseline=0.5cm,yscale=0.7,xscale=0.8]
\draw[thick](0,0)to[out=90,in=-140]node[tv,sloped,pos=0.4](t){}(1,1.75);
\draw[thick](1,0)--(1,0.25)to[out=140,in=-30](t.south);
\draw[thick](1,0.25)to[out=60,in=-60](1,1.75)--(1,2);
\end{tikzpicture}
=p.
$$ 

It remains to check that $\cramped{N_{H_N}}=N$. The equality is true for morphisms from $\ms{Com}$ by Proposition \ref{prop:comHopf} and one easily checks the equality for a single chord. These morphisms generate $\ms{iCom}(R)$, and so the two functors are equal.

\section{Semicommutative Hopf algebroids}

The main idea of this paper was that groups (or commutative Hopf algebras) have symmetric nerves, and that deforming this symmetric structure to a braided structure provides their quantization. A natural idea is to extend it to other objects having symmetric nerves. The simplest option is to generalize groups to groupoids. The quantum object that we get is ``quantum groupoids with a classical base''.

\begin{defn}[{Maltsionitis \cite[{$\S$5}]{Mal}}]
A \emph{semicommutative Hopf algebroid} over a commutative ring $R$ is a pair of $R$-algebras $B$ and $H$, with $B$ commutative, with the following additional structure: 
\begin{itemize}
\item two $R$-algebra homomorphisms $\eta_L,\eta_R\colon B\to Z(H)$ (
the center of $H$); this makes $H$ to a $B$-$B$-bimodule (via $b_1\otimes h\otimes b_2\mapsto \eta_L(b_1)\,h\,\eta_R(b_2)$)
\item  $B$-$B$-bimodule morphisms $\Delta\colon H\to H\otimes_B H$ and  $\epsilon\colon H\to B$ making $H$ to a coalgebra in the monoidal category of $B$-$B$-bimodules, which are also $R$-algebra morphisms
\item an invertible $R$-algebra anti-homomorphism $S\colon H\to H$ such that 
\begin{gather*}
S\circ\eta_L=\eta_R\quad S\circ\eta_R=\eta_L\\
m\circ(\id_H\otimes_B S)\circ\Delta=\epsilon\circ\eta_L,\quad m\circ(S\otimes_B\id_H)\circ\Delta=\epsilon\circ\eta_R
\end{gather*}
where $m$ is the product on $H$.
\end{itemize}
\end{defn}

\begin{rem}
For simplicity we gave the definition in the category of $R$-modules. The definition and also the rest of this section can be generalized to BMCs where $H\otimes_B H$ and its iterations are well defined and behaved. The only change that needs to be done is that $\eta_L\colon B\to H$ should be a left-central and $\eta_R\colon B\to H$ a right central morphism of algebras.
\end{rem}

A \emph{commutative Hopf algebroid} corresponds to the case when $H$ is commutative. A
 \emph{semicommutative Poisson Hopf algebroid} is a commutative Hopf algebroid together with a Poisson bracket on $H$ with the properties that the maps $\eta_{L,R}$ send $B$ to the Poisson center of $H$ and $\Delta\colon H\to H\otimes_B H$ and $\epsilon\colon H\to B$ are Poisson algebra morphisms, where the Poisson bracket on $B$ is zero. The \emph{quantization problem} is to deform $m$, $\Delta$, and $S$ (the rest of the structure is not deformed) so that we obtain a semicommutative Hopf algebroid and so that the deformed $m$ is a quantization of the Poisson bracket on $H$.

This problem can again be solved using nerves. If $N\colon \ms{BrCom}\to R\text{-mod}$ or $N\colon \ms{Com}\to R\text{-mod}$ is a braided lax monoidal functor then $N(\bullet)\in R\text{-mod}$ is a commutative algebra and $N({\bullet\bullet})$ is its bimodule via the $N(\bullet)$-actions 
$$
\begin{tikzpicture}[scale=0.35, baseline=0.5cm]
\fill[black!10]%
     (1.5,3)to[out=-90,in=90](0.5,0)--
     (1.5,0)..controls +(0.1,1.5) and +(-0.1,1.5)..(2.5,0)--
     (4.5,0)to[out=90,in=-90](3.5,3)--cycle;
\draw[black!30,thick]%
     (1.5,3)to[out=-90,in=90](0.5,0)
     (1.5,0)..controls +(0.1,1.5) and +(-0.1,1.5)..(2.5,0)
     (4.5,0)to[out=90,in=-90](3.5,3);
\draw (1,0)node[bul]{}to[out=90,in=-90]+(1,3)node[bul]{};
\draw (3,0)node[bul]{}to[out=90,in=-90]+(-1,3)node[bul]{};
\draw (4,0)node[bul]{}to[out=90,in=-90]+(-1,3)node[bul]{};
\end{tikzpicture}
\qquad\text{and}\qquad
\begin{tikzpicture}[scale=0.35, baseline=0.5cm]
\fill[black!10]%
     (0.5,3)to[out=-90,in=90](-0.5,0)--
     (1.5,0)..controls +(0.1,1.5) and +(-0.1,1.5)..(2.5,0)--
     (3.5,0)to[out=90,in=-90](2.5,3)--cycle;
\draw[black!30,thick]%
     (0.5,3)to[out=-90,in=90](-0.5,0)
     (1.5,0)..controls +(0.1,1.5) and +(-0.1,1.5)..(2.5,0)
     (3.5,0)to[out=90,in=-90](2.5,3);
\draw (0,0)node[bul]{}to[out=90,in=-90]+(1,3)node[bul]{};
\draw (1,0)node[bul]{}to[out=90,in=-90]+(1,3)node[bul]{};
\draw (3,0)node[bul]{}to[out=90,in=-90]+(-1,3)node[bul]{};
\end{tikzpicture}
$$
The map \eqref{Segal} is easily seen to factor through 
\begin{equation}\label{Seg2}
\underbrace{N({\bullet\bullet})\otimes_{N(\bullet)}N({\bullet\bullet})\otimes_{N(\bullet)}\dots \otimes_{N(\bullet)} N({\bullet\bullet})}_{n-1} \to N(\bullet^{n}).
\end{equation}
We shall say that $N$ \emph{satisfies the Segal condition} \cite{Seg} (or \emph{groupoid nerve condition}) if \eqref{Seg2} is an isomorphism for every $n$ and if $c^N_1\colon R\to N(1)$ is an isomorphism.

\begin{thm}\label{thm:nervoid}
The category of semicommutative Hopf algebroids over $R$  with invertible antipodes is equivalent to the category of braided lax monoidal functors
$$N\colon \ms{BrCom}\to R\text{-}\mathrm{mod}$$
satisfying the Segal condition. The same is true for commutative Hopf algebroids and symmetric lax monoidal functors 
$$N\colon \ms{Com}\to R\text{-}\mathrm{mod}.$$

The Hopf algebroid corresponding to $N$ is 
$$H=N({\bullet\bullet})\qquad B=N(\bullet).$$
 The algebra structure on $H$ and $B$ comes from the algebra structure of ${\bullet\bullet}$ and $\bullet$.
The remaining operations are given by
$$
\Delta=\begin{tikzpicture}[xscale=0.3,yscale=0.6,baseline=0.5cm]
\coordinate (bor) at (0.8,0);
\coordinate(1) at (-1,0);
\coordinate(2) at (1,0);
\coordinate(a) at (-2,2);
\coordinate(b) at (0,2);
\coordinate(c) at (2,2);
\fill[black!10] ($(1)-(bor)$)to[out=90,in=-90]($(a)-(bor)$)--($(c)+(bor)$)to[out=-90,in=90]($(2)+(bor)$);
\draw[black!30,thick]($(1)-(bor)$)to[out=90,in=-90]($(a)-(bor)$) ($(c)+(bor)$)to[out=-90,in=90]($(2)+(bor)$);
\node(1) at (1)[bul]{};
\node(2) at (2)[bul]{};
\node(a) at (a)[bul]{};
\node(b) at (b)[bul]{};
\node(c) at (c)[bul]{};
\draw (1)to[out=90,in=-90](a) (2)to[out=90,in=-90](c);
\end{tikzpicture}\qquad
\epsilon=\begin{tikzpicture}[scale=0.6,baseline=0.5cm]
\fill[black!10] (-0.5,0)to[out=90,in=-90](0,2)--(1,2)to[out=-90,in=90](1.5,0);
\draw[black!30,thick](-0.5,0)to[out=90,in=-90](0,2)  (1,2)to[out=-90,in=90](1.5,0);
\node(1) [bul] at (0,0){};
\node(2) [bul] at (1,0){};
\node(3) [bul] at (0.5,2){};
\draw (1) to[out=90,in=-90](3) (2) to[out=90,in=-90](3);
\end{tikzpicture}\qquad
S=\begin{tikzpicture}[scale=0.6,baseline=0.5cm]
\fill[black!10](-0.3,0)--(-0.3,2)--(1.3,2)--(1.3,0);
\draw[black!30,thick](-0.3,0)--(-0.3,2) (1.3,2)--(1.3,0);
\node(1) at (0,0)[bul]{};
\node(2) at (1,0)[bul]{};
\node(a) at (0,2)[bul]{};
\node(b) at (1,2)[bul]{};
\draw (2)..controls +(0,1) and +(0,-1)..(a);
\draw[line width=1ex,black!10] (1)..controls +(0,1) and +(0,-1)..(b);
\draw (1)..controls +(0,1) and +(0,-1)..(b);
\end{tikzpicture}
$$
$$\eta_L=\begin{tikzpicture}[scale=0.6,baseline=0.5cm]
\fill[black!10] (-0.5,0)to[out=90,in=-90](-1,2)--(1,2)to[out=-90,in=90](0.5,0);
\draw[black!30,thick](-0.5,0)to[out=90,in=-90](-1,2) (1,2)to[out=-90,in=90](0.5,0);
\node(1) [bul] at (0,0){};
\node(2) [bul] at (-0.5,2){};
\node(3) [bul] at (0.5,2){};
\draw (1)to[out=90,in=-90](2);
\end{tikzpicture}\qquad
\eta_R=\begin{tikzpicture}[scale=0.6,baseline=0.5cm]
\fill[black!10] (-0.5,0)to[out=90,in=-90](-1,2)--(1,2)to[out=-90,in=90](0.5,0);
\draw[black!30,thick](-0.5,0)to[out=90,in=-90](-1,2) (1,2)to[out=-90,in=90](0.5,0);
\node(1) [bul] at (0,0){};
\node(2) [bul] at (-0.5,2){};
\node(3) [bul] at (0.5,2){};
\draw (1)to[out=90,in=-90](3);
\end{tikzpicture}$$
where we implicitly use the isomorphism $N(\bullet^3)\cong N({\bullet\bullet})\otimes_{N(\bullet)}N({\bullet\bullet})$ given by the Segal condition. 

Finally, semicommutative Poisson Hopf algebroids over $R$ are equivalent to i-braided lax monoidal functors
$$N\colon \ms{iCom}(R)\to R\text{-}\mathrm{mod}$$
satisfying the Segal condition. The Poisson bracket on $H$ comes from the Poisson bracket on ${\bullet\bullet}$.
\end{thm}

The theorem can be proven by a suitable modification of the proofs in \S\ref{sec:proofs}; we leave the details to the reader. Let us just mention a few hints. The algebra $B=N(\bullet)$ must be commutative since $\bullet\in\ms{BrCom}$ is commutative and $N$ is braided lax monoidal. Similarly one checks that $\eta_L$ and $\eta_R$ send $B=N(\bullet)$ to the center of $H=N({\bullet\bullet})$. The proof of Theorem \ref{thm:brnerve} (\S\ref{sec:pfbr}) generalizes without any problems, if we set
$$N(\bullet^n)=\underbrace{H\otimes_B H\otimes_B\dots \otimes_B H}_{n-1}$$
where $\otimes_B$ is given in terms of the $B$-bimodule structure of $H$ given by $\eta_L$ and $\eta_R$. Likewise, the proof of Theorem \ref{thm:inerve} (\S\ref{sec:pfi}) goes through by putting
$$F(\bullet^n)=
\Bigl({\bigotimes}_B\Bigr)_{i=1}^n H$$
where this time $\otimes_B$ is defined using $\eta_L$ only.
\medskip 

If $\Q\subset R$, we get from Theorem \ref{thm:nervoid} immediately a solution of the quantization problem:  If $N\colon \ms{iCom}(R)\to R\text{-}\mathrm{mod}$ is the nerve of a semicommutative Poisson Hopf algebroid then the composition of braided lax monoidal functors
$$\ms{BrCom}\xrightarrow{U_\Phi}\ms{iCom}(\Q)_\hbar^\Phi\subset\ms{iCom}(R)_\hbar^\Phi\xrightarrow{N_\hbar} R\text{-}\mathrm{mod}_\hbar$$
satisfies the Segal condition and thus gives us a semicommutative Hopf algebroid.

An interesting question remains whether our method can be applied for quantization of other objects with symmetric nerves (higher groups or groupoids).

\end{document}